\documentclass[envcountsect]{llncs}
\usepackage{cite}
\usepackage[caption=false]{subfig}
\usepackage{pgfplots}
\pgfplotsset{compat=1.12}
\usepackage{geometry}
\usepackage{indentfirst}
\usepackage{indentfirst}
\usepackage{changepage}
\usepackage{amsmath}
\usepackage{wrapfig}
\usepackage{amssymb}
\usepackage{graphicx} 
\usepackage{algorithm}
\usepackage{algpseudocode}
\usepackage{tikz}
\usepackage[T1]{fontenc}
\usepackage{authblk}
\usepackage{verbatim}
\usepackage{xspace}
\usetikzlibrary{arrows,shapes}
\usepackage[caption=false]{subfig}

\newtheorem{cl}{Claim}[section]
\newtheorem{prop}{Proposition}[section]

\numberwithin{claimcounter}{cl}
\newcommand{\myshortestpath}	{{path-dominated shortest path}\xspace}
\newcommand{\K}				{{\mathcal{K}}}
\newcommand{\reals}			{{\mathbb{R}}}
\newcommand{\mysizeF}			{{tight-size function}\xspace}
\newcommand{\MysizeF}			{{Tight-size function}\xspace}
\newcommand{\myS}				{{\mu}}

\begin{document}
\title{ Efficient algorithms for computing a minimal homology basis}
\author{Tamal K. Dey \inst{1}\footnote{\email{tamaldey@cse.ohio-state.edu}} \and Tianqi Li \inst{1}\footnote{\email{li.6108@osu.edu}} \and Yusu Wang\inst{1}\footnote{\email{yusu@cse.ohio-state.edu}}}
\institute{
Department of Computer Science and Engineering, The Ohio State University, Columbus}
\maketitle
\begin{abstract}

    Efficient computation of shortest cycles which form a homology basis under $\mathbb{Z}_2$-additions in a given simplicial complex $\mathcal{K}$ has been researched actively in recent years. When the complex $\mathcal{K}$ is a weighted graph with $n$ vertices and $m$ edges, the problem of computing a shortest (homology) cycle basis is known to be solvable in $O(m^2n/\log n+ n^2m)$-time. Several works \cite{borradaile2017minimum, greedy} have addressed the case when the complex $\mathcal{K}$ is a $2$-manifold. The complexity of these algorithms depends on the rank  $g$ of the one-dimensional homology group of $\mathcal{K}$. This rank $g$ has a lower bound of $\Theta(n)$, where $n$ denotes the number of simplices in $\mathcal{K}$, giving an $O(n^4)$ worst-case time complexity for the algorithms in \cite{borradaile2017minimum,greedy}. This worst-case complexity is improved in \cite{annotation} to  $O(n^\omega + n^2g^{\omega-1})$ for general simplicial complexes where $\omega< 2.3728639$ \cite{le2014powers} is the matrix multiplication  exponent. Taking $g=\Theta(n)$,  this provides an $O(n^{\omega+1})$ worst-case algorithm. In this paper, we improve this time complexity. Combining the divide and conquer technique from \cite{DivideConquer} with the use of annotations from \cite{annotation}, we present an algorithm that runs in $O(n^\omega+n^2g)$ time giving the first $O(n^3)$ worst-case algorithm for general complexes. If instead of minimal basis, we settle for an  approximate basis, we can improve the running time even further. We show that a $2$-approximate minimal homology basis can be computed in $O(n^{\omega}\sqrt{n \log  n})$ expected time. We also study more general measures for defining the minimal basis and identify reasonable conditions on these measures that allow computing a minimal basis efficiently. %For higher dimensional cycles(with dimension $d\geq 1$), we improve the running time for computing a minimum homology basis from $O(\beta_dn^4)$ to $O(n^{\omega+1})$ under a size function defined in \cite{chen2010measuring} where $\beta_d$ is the rank of $d$-dimensional homology group.

\end{abstract}

\section{Introduction}
Many applications in science and engineering require computing ``features'' in a shape that is finitely represented by a simplicial complex. These features sometimes include topological features such as ``holes'' and ``tunnels'' present in the shape. A concise definition of these otherwise vague notions can be obtained by considering homology groups and their representative cycles. In particular, a one-dimensional homology basis, that is,
%a maximally independent set of cycles in the $1$-skeleton of the input simplicial complex can be taken as a representative of the ``holes" and ``tunnels" present in the shape.
a set of independent cycles in the $1$-skeleton of the input simplicial complex whose homology classes form a basis for the first homology group, can be taken as a representative of the ``holes'' and ``tunnels'' present in the shape. However, instead of any basis, one would like to have a homology basis whose representative cycles are small under some suitable metric, thus bringing the `geometry' into picture along with topology. 

When the input complex is a graph with $n$ vertices and $m$ edges, the homology basis coincides with what is called the cycle basis and its minimality is measured with respect to the total weights of the cycles assuming non-negative weights on the edges. A number of efficient algorithms have been designed to compute such a minimal cycle basis for a weighted graph~\cite{borradaile2017minimum,depina,horton1987polynomial,DivideConquer, label}. The best known algorithm for this case runs in $O(m^2n/\log n+ n^2m)$ \cite{label}. 

When the input is a simplicial complex, one dimensional homology basis is determined by the simplices of dimension up to 2. Thus, without loss of generality, we can assume that the complex has dimension at most $2$, that is, it consists of vertices, edges, and triangles. The $1$-skeleton of the complex is a graph (weighted if the edges are). Therefore, one can consider a minimal cycle basis in the $1$-skeleton. However, the presence of triangles makes some of these basis elements to be trivial in the
homology basis. Therefore, the computation of the minimal homology basis in a simplicial complex differs from the minimal cycle basis in a graph.
In this paper, we show that the efficient algorithms of \cite{DivideConquer} for computing \emph{a minimal cycle basis} can be adapted to computing \emph{a minimal homology basis} in a simplicial complex (by combining with an algorithm \cite{annotation} to compute the so-called annotations). 
%In this paper, we show that the efficient algorithms \yusu{Tianqi, Refs here} for computing a minimal cycle basis can be adapted using an algorithm to compute so-called {\em annotations} \cite{annotation} to compute a minimal homology basis in a simplicial complex. 
In the process we improve the current best time complexity bound for computing a minimal homology basis and also extend these results to more generalized measures.

More specifically, for the special case of a combinatorial $2$-manifold with weights on the edges,
Erickson and Whittlesey \cite{greedy} gave an $O(n^2\log n+gn^2+g^3n)$-time algorithm to compute a minimal homology basis where $n$ is the total number of simplices and $g$ is the rank of the first homology group. Dey et al.~\cite{dey2010approximating} and Chen and Friedman~\cite{chen2010measuring} generalized the results above to arbitrary simplicial complexes. Busaryev et al. \cite{annotation} improved the running time of this generalization from $O(n^4)$ \cite{dey2010approximating} to $O(n^{\omega}+n^2g^{\omega-1})$ where $\omega< 2.3728639$~\cite{le2014powers} is the matrix multiplication exponent. This gives the best known $O(n^{1+\omega})$ worst-case time algorithm when $g=\Theta(n)$. 
In Section \ref{sec:h1}, combining the divide and conquer approach of \cite{DivideConquer} with the use of annotations~\cite{annotation}, we %obtain an $O(n^2g + n^\omega)$ time algorithm which is $O(n^3)$ in the worst case when $g=\Omega(n)$.
develop an improved algorithm to compute a minimal 1-dimensional homology basis for an arbitrary simplicial complex in only $O(n^2 g + n^\omega)$ time. Considering $g=\Theta(n)$, this gives the first $O(n^3)$ worst-case time algorithm for the problem.

%As for high dimension $d$ with $d>1$ in which computing a minimum homology basis under weight is NP-hard \cite{chen2010hardness}, Chao Chen and Daniel Freedman \cite{chen2010measuring} took a different measure which is the smallest geodesic ball that contains the cycle. Under this size, an algorithm to compute a minimum homology basis in any dimension $d$ exists with time complexity $O(\beta_d n^4)$ where $\beta_d$ is the rank of the $d$-dimensional homology group $\mathsf{H}_d$. 

%We first present an algorithm to compute a 1-dimensional minimum homology basis which leads to improved time bounds $O(n^{\omega}+n^2g)$. At a high level, our algorithm is based on a divide-and-conquer technique from \cite{DivideConquer} which used the idea of de Pina \cite{depina}. In the algorithm, we always maintain a set of support vectors which forms a basis of the subspace orthogonal to the space generated by those already computed cycles. Every time when we compute a new cycle, a minimum cycle that is nonorthogonal to the given support vector $S$ can be found. After that, we continue to update the support vectors to satisfy the``orthogonal" condition. The final output will be a minimum homology basis.

%Instead of a minimum homology basis, if we settle for an approximation, we can further improve the time complexity. 
We can further improve the time complexity if we allow for approximations. An algorithm to compute a 2-approximate minimal homology basis are given in Section \ref{sec:approximation} running in $O(n^{\omega}\sqrt{n\log n})$ expected time. 
% in which we first compute a 2-approximate minimum cycle basis and then extract a homology basis which is proved to be a 2-approximate minimum homology basis.

All of the above algorithms operate by computing a set of candidate cycles that necessarily includes at least one minimal homology basis and then selecting one of these minimal bases. The standard proof \cite{greedy} of the fact that the candidate set includes a minimal basis uses the specific distance function based on the shortest path metric and a size function that assigns total weight of the edges in a cycle as its size. In Section \ref{sec:generalization}, we identify general conditions for the distance and size function so that the divide and conquer algorithm still works without degrading in time complexity. This allows us to consider distance function beyond the shortest path metric and the size function beyond the total weight of edges as we illustrate with two examples.
Specifically, we can now compute a minimal homology basis whose size is induced by a general map $F:\K\to Z$ for any metric space $Z$. 
%We then notice that the divide-and-conquer technique is actually aiming at finding a basis in a given set. Thus it is possible to extend the distance and the size function so that under such a distance and size function, the algorithm still works well. We give such an extension and two examples of such distances and sizes.

%In high dimension, it is known from \cite{chen2010hardness} that computing a minimum homology basis under volume as the size is NP-hard. It follows from \cite{chen2010measuring} that there is a measure of cycles under which an algorithm to compute a minimum homology basis in polynomial time exists, in time $O(\beta_dn^4)$ where $\beta_d$ is the rank of $d$-dimensional homology group $\mathsf{H}_d$. We improved this algorithm, using persistent algorithm \cite{edelsbrunner2008persistent} as well as annotations for $d$-simplices \cite{annotation}, so that our time complexity is $O(n^{\omega+1})$ which is better when $\beta_d=\Omega(n)$.

%%%%%%%%%BACKGROUND%%%%%%%%%%
\section{Background and notations}
\label{sec:notations}

\noindent In this paper, we are interested in computing a minimal basis for the 1-dimensional homology group of a simplicial complex over the field $\mathbb{Z}_2$. In this section we briefly introduce some relevant concepts here; the details appear in standard books on algebraic topology such as \cite{hatcher2002algebraic}.

\paragraph{Homology.} Let $\mathcal{K}$ be a connected simplicial complex. A $d$-chain $c$ is a formal sum, $c =\sum a_i\sigma_i$ where the $\sigma_i$s are the $d$-simplices of $\K$ and the $a_i$s are the coefficients with $a_i\in \mathbb{Z}_2$. We use $\mathsf{C}_d$ to denote the group of $d$-chains which is formed by the set of $d$-chains together with the addition. Note that there is a one-to-one correspondence between the chain group $\mathsf{C}_d$ and the family of subsets of $\mathcal{K}_d$ where $\mathcal{K}_d$ is the set of all $d$-simplices. Thus $\mathsf{C}_d$ is isomorphic to the space $(\mathbb{Z}_2)^{n_d}$ where $n_d$ is the number of $d$-simplices in $\K$. Naturally all $d$-simplices in $\mathcal{K}$ form a basis of $\mathsf{C}_d$ in which the $i$-th bit of the coordinate vector of a $d$-chain indicates whether the corresponding $d$-simplex appears in the chain.

The boundary of a $d$-simplex is the sum of all its $(d-1)$-faces. This can be interpreted and extended to a $d$-chain as a \textit{boundary map} $\partial_d : \mathsf{C}_d \to  \mathsf{C}_{d-1}$, where the boundary of a chain is defined as the sum of the boundaries of its simplices. A $d$-cycle $c$ is a $d$-chain with empty boundary, $\partial_d c=0$. Since $\partial_d$ commutes with addition, we have the group of $d$-cycles, $\mathsf{Z}_d$, which is the kernel of $\partial_d$, $\mathsf{Z}_d:=ker{\partial_d}$. 
A $d$-boundary $c$ is a $d$-chain that is the boundary of a $(d + 1)$-chain, $c =\partial_{d+1} b$ for some $b\in \mathsf{C}_{d+1}$. The group of $d$-boundaries $\mathsf{B}_d$ is the image of $\partial_{d+1}$, that is, $\mathsf{B}_d := im \partial_{d+1}$. Notice that $\mathsf{B}_d$  is a subgroup of $\mathsf{Z}_d$. Hence we can consider the quotient $\mathsf{Z}_d/\mathsf{B}_d$ which constitutes the $d$-dimensional homology group denoted as $\mathsf{H}_d$. Each element in $\mathsf{H}_d$, called a homology class, is an equivalence class of $d$-cycles whose difference is always in $\mathsf{B}_d$. Two cycles are said to be $homologous$ if they are in the same homology class.
% We denote the $d$-th $Betti$ $number$, the dimension of the $\mathsf{H}_d$, as $\beta_d$ and the $1$-st $Betti$ $number$ as $g$ in particular.

%The groups $\mathsf{Z}_d$, $\mathsf{B}_d$ and $\mathsf{H}_d$ defined over $\mathbb{Z}_2$ form vector spaces and hence admit bases called cycle basis, boundary basis and homology basis. We are concerned with the homology bases of $\mathsf{H}_d$ and particularly in $\mathsf{H}_1$. The cardinality of a basis of $\mathsf{H}_d$ is called the $d$-th {\em Betti number} of $\mathcal{K}$. We use $g$ to denote the $1$-st Betti number which is the dimension $\mathsf{H}_1$ when interpreted as  a vector space.

%\yusu{Probably putting def. of homology basis, min homology basis here more formally.} 
%{Tamal: I added the definitions, check!}\\
Under $\mathbb{Z}_2$ coefficients, the groups $\mathsf{C}_d$, $\mathsf{Z}_d$, $\mathsf{B}_d$ and $\mathsf{H}_d$ are all vector spaces.
A basis of a vector space is a set of vectors of minimal cardinality that generates the entire vector space. 
We are concerned with the homology bases of $\mathsf{H}_d$ and particularly in $\mathsf{H}_1$ (more formally below). 
%The cardinality of a basis of $\mathsf{H}_d$ is called the $d$-th {\em Betti number} of $\mathcal{K}$. 
We use $L = rank(\mathsf{Z}_1)$ to denote the dimension of vector space $\mathsf{Z}_1$ and use $g = rank(\mathsf{H}_1)$ to denote the \emph{$1$-st Betti number} of $\K$, which is the dimension of vector space $\mathsf{H}_1$. 
% when interpreted as  a vector space.
\begin{itemize}
\item 	A set of cycles $C_1,\cdots, C_L$, with $L = rank(\mathsf{Z}_1)$, that generates the cycle space $\mathsf{Z}_1$ is called its {\em cycle basis}. 
\item For any $1$-cycle $c$, let $[c]$ denote its homology class. A set of homology classes $\{[C_1], . . . , [C_g]\}$ that constitutes a basis of $\mathsf{H}_1$ is called a {\em homology basis}. For simplicity, we also say a set of cycles $\{C_1, C_2, \cdots, C_g\}$ is a homology basis if their corresponding homology classes $[C_1], [C_2], \cdots, [C_g]$ form a basis for $\mathsf{H}_1(\mathcal{K})$.
\item Let $\myS: \mathsf{Z}_1\rightarrow \mathbb{R}^+\cup{\{0\}}$ be a size function that assigns a non-negative weight to each cycle $C\in \mathsf{Z}_1$. A cycle or homology basis $C_1,\cdots,C_l$ is called {\em minimal} if $\sum_{i=1}^l \myS(C_i)$ is minimal among all bases of $\mathsf{Z}_1$ ($l=L$) or $\mathsf{H}_1(\mathcal{K})$ ($l=g$) respectively.
\end{itemize} 
%Earliest Basis
%\textbf{Earliest Basis.}
%Our approaches later would use the following concept from \cite{annotation}.
%\begin{definition}
%\rm{(Earliest Basis)}. \textit{Given a matrix $A$ with rank $r$, the set of columns $B_{opt} = \{a_{i_1}, \cdots , a_{i_r} \}$ is called the \textbf{earliest basis} if the column indices $\{i_1,\ldots , i_r\}$ are the lexicographically smallest index set such that the corresponding columns of $A$ have full rank.}
%\end{definition}

%%A convenient view to consider the earliest basis is that a column vector $a$ of a matrix $A$ is in the earliest basis if and only if it does not belong to the subspace generated by those column vectors in $A$ to its left. 

%\begin{prop}\rm{\cite{annotation}} \label{prop:annotation}
%\textit{Let $A$ be an $m\times n$ matrix of rank $r$ with entries over $\mathbb{Z}_2$ where $n \leq m$, then there is an $O(m n^{\omega-1})$ time algorithm to compute the earliest basis $B_{opt}$ of $A$.}
%\end{prop}

\paragraph{Annotation.} To compute a minimal homology basis of a simplicial complex $\mathcal{K}$, it is necessary to have a way to represent and distinguish homology classes of cycles. Annotated simplices have been used for this purpose in earlier works: For example, Erickson and Wittlesey~\cite{greedy} and Borradaile et al.~\cite{borradaile2017minimum} used them for computing optimal homology cycles in surface embedded graphs. Here we use a version termed as {\em annotation} from \cite{annotation} which gives an algorithm to compute them in matrix multiplication time for general simplicial complexes. An annotation for a $d$-simplex is a $g$-bit binary vector, where $g = rank(\mathsf{H}_d(\K))$. The annotation of a cycle $z$, which is the sum of annotations of all simplices in $z$, provides the coordinate vector of the homology class of $z$ in a pre-determined homology basis. More formally, 
\begin{definition}[Annotation] \label{TimeAnnotation}
\textit{Let $\K$ be a simplicial complex and $\K_d$ be the set of $d$-simplices in $\K$. An annotation for $d$-simplices is a function $a: \mathcal{K}_d \to (\mathbb{Z}_2)^{g}$ with the following property: any two $d$-cycles $z$ and $z'$ are homologous if and only if $$\sum_{\sigma\in z}a(\sigma)=\sum_{\sigma\in z'}a(\sigma)$$
Given an annotation $a$, the annotation of any $d$-cycle $z$ is defined by $a(z)=\sum_{\sigma\in z}a(\sigma)$.}
\end{definition}
\begin{prop}[\cite{annotation}]
There is an algorithm that annotates the $d$-simplices in a simplicial complex with $n$ simplices in $O(n^{\omega})$ time.
\end{prop}

%%%%%%%%SECTION3%%%%%%%%%%%%%%%%%%%%
\section{Minimal homology basis} \label{sec:h1}
\noindent In this section, we describe an efficient algorithm to compute a minimal homology basis of the 1-dimensional homology group $\mathsf{H}_1(\K)$. The algorithm uses the divide and conquer technique from \cite{DivideConquer} where they compute a minimal \emph{cycle} basis in a weighted graph. The authors in~\cite{borradaile2017minimum} adapted it for computing optimal homology basis in surface embedded graphs. We adapt it here to simplicial complexes using edge annotations~\cite{annotation}.

More specifically, let $\mathcal{K}$ be a simplicial complex with $n$ simplices -- Since we are only interested in 1-dimensional homology basis, it is sufficient to consider all simplices with dimension up to 2, namely vertices, edges, and triangles. Hence we assume that $\mathcal{K}$ contains only simplices of dimension at most 2.
%The homology group $\mathsf{H}_1$ admits a basis of size $g=rank(\mathsf{H}_1(\mathcal{K}))$ which is a set of $g$ homology classes $[C_i]$ of 1-dimensional cycles $C_i$.
Assume that the edges in $\mathcal{K}$ are weighted with non-negative weights.
Given any homology basis $\{C_1, \ldots, C_g\}$ where $g = rank(\mathsf{H}_1(\K))$, we define the \emph{size $\myS(C)$ of a cycle $C\in Z_1(\mathcal{K})$} as the \emph{total weights} of its edges. As defined in Section \ref{sec:notations}, the problem of computing a minimal homology basis of $\mathsf{H}_1$ is now to find a basis $\mathcal{C}=\{C_1, C_2, \cdots, C_g\}$ such that the sum of $\sum_{i=1}^g \myS(C_i)$ is the smallest.\\
%
%Since we are only interested in 1-cycles, it is sufficient to consider all simplices with dimension up to 2, namely vertices, edges, and triangles. Hence we assume the dimension $dim$ of the complex $\mathcal{K}$ is at most 2.
%

The high-level algorithm to compute such a minimal homology basis of $\mathsf{H}_1$ group proceeds as follows. %Algorithm \ref{alg} presents. 
First, we need to annotate all 1-simplices implemented by the algorithm of \cite{annotation}. Then we compute a candidate set of cycles which includes a minimal homology basis. At last, we extract such a minimal homology basis from the candidate set. 
%step 1 of the algorithm (Line 2) is implemented by the algorithm of \cite{annotation}.

%\begin{algorithm}
%\caption{Compute a minimal homology basis}
%\label{alg}
% \begin{algorithmic}[1] 
% 	\Procedure{HomologyBasis}{$\mathcal{K}$}
%	\State Annotate all 1-simplices in $\mathcal{K}$.
%	\State Generate a candidate set $\mathcal{G}$ of cycles which includes a minimal homology basis.
%	\State Call \Call{CycleBasis}{$\mathcal{G}$} to find a minimal homology basis in the candidate set.
%	\EndProcedure
%\end{algorithmic}
%\end{algorithm}

%candidate set
\paragraph{Candidate set.}
We now describe the step to compute a candidate set $\mathcal{G}$ of cycles that contains a minimal homology basis. We use the shortest path tree approach which dates back to Horton's algorithm for a minimal cycle basis of a graph~\cite{horton1987polynomial}. It was also applied in other earlier works, e.g. \cite{dey2010approximating, greedy}. We first generate a candidate set $\mathcal{G}(p)$ for every vertex $p\in vert(\mathcal{K})$, where $vert(\K)$ is the set of vertices of $\K$. Then we take the union of all $\mathcal{G}(p)$ and denote as $\mathcal{G}$, i.e. $\mathcal{G}=\cup_{p\in vert(\K)} \mathcal{G}(p)$. To compute $\mathcal{G}(p)$, first we construct a shortest path tree $T_p$ rooted at $p$. Let $\Pi_p(u, v)$ denote the unique path connecting two vertices $u$ and $v$ in $T_p$. Then each nontree edge $e=(u, v)$ generates a cycle $C(p, e)=e \circ \Pi_p(u, v)$. The union of all such cycles constitutes the candidate set of the vertex $p$, i.e. $\mathcal{G}(p)=\cup_{e\in edge(\K)\setminus E_{p}}C(p, e)$ where $E_{p}$ is the set of tree edges in $T_p$. Note that the number of cycles in $\mathcal{G}(p)$ is $O(n)$ for each vertex $p \in vert(\K)$. Hence there are $O(n^2)$ candidate cycles in $\mathcal{G}$ in total. They, together with their sizes, can be computed in $O(n^2\log n)$ time. 

\begin{proposition}[\cite{dey2010approximating, greedy}] \label{candidate}
The candidate set $\mathcal{G}$ has $O(n^2)$ cycles and admits a minimal homology basis.
\end{proposition}

%compute basis
\subsection{Computing a minimal homology basis}\label{sec:h1b}
What remains is to compute a minimal homology basis from the candidate set $\mathcal{G}$. 
To achieve it, we modify the divide and conquer approach from \cite{DivideConquer} which improved the algorithm of \cite{depina} for computing a minimal cycle basis of a graph with non-negative weights. 

This approach uses an auxiliary set of support vectors \cite{DivideConquer} that helps select a minimal homology basis from a larger set containing at least one minimal basis; in our case, this larger set is $\mathcal{G}$. %The candidate set $\mathcal G$ in our case will be such a set.

A support vector $S$ is a vector in the space of $g$-dimensional binary vectors $\mathcal{S}=\{0, 1\}^g$. The use of support vectors along with annotations requires us to perform more operations  without increasing the complexity of the divide and conquer approach. Let $a(C)$ denote the annotation of a cycle $C$. First, we define the function: 
%demands some more calculations which we perform without increasing the complexity of the divide and conquer approach. Define a function 
$$m: \mathcal{S}\times \mathcal{G} \to \{0, 1\} \mbox{ with } m(S, C)=\langle S, a(C)\rangle \mbox{ where } \langle \cdot , \cdot \rangle \mbox{ is the inner product over } \mathbb{Z}_2.$$ We say a cycle $C$ is {\it orthogonal} to a support vector $S_i$ if $m(S_i, C)=0$ and is {\em non-orthogonal} if $m(S_i, C)=1$.
We would choose cycles $C_1$, $\cdots$, $C_g$ iteratively from a set guaranteed to contain a minimal homology basis and add them to the minimal homology basis. During the procedure, the algorithm always maintains a set of support vectors $S_1, S_2, \cdots, S_g$ with the following properties: 
\begin{itemize}
\item[(1).] $S_1, S_2, \cdots, S_g$ form a basis of $\{0, 1\}^g$. 
\item[(2).] If $C_1, C_2, \cdots, C_{i-1}$ have already been computed, $m(S_i, C_j)=0$, $j<i$. 
\end{itemize}

Suppose that in addition to properties (1) and (2), we have the following additional condition to choose $C_i$s, then the set $C_1,C_2,\ldots,C_g$ constitutes a minimal homology basis.

\begin{itemize}
	\item[(3).] If $C_1,C_2,\cdots, C_{i-1}$ have already been computed, $C_i$ is chosen so that
	$C_i$ is the shortest cycle with $m(S_i,C_i)=1$.
\end{itemize}

%Of course, property (3) contradicts property (2) if we keep the same support vectors. 
If we keep the same support vectors, after we select a new cycle $C_i$, $m(S_{i+1}, C_i)=0$ may not hold which means the property (2) may not hold.
Therefore, we update the support vectors $S_{i+1}, \cdots, S_g$ after computing $C_i$ so that the orthogonality condition (2) holds.  If chosen with condition (3), the cycle $C_i$ becomes independent of the cycles previously chosen as stated below: 

\begin{cl}\label{independent}
	For any $i\leq g$, if property (1) and (2) hold, then for any cycle $C$ with $m(S_i, C)=1$, $[C]$ is independent of $[C_1], [C_2], \cdots, [C_{i-1}]$. 
\end{cl}
\begin{proof}
	By property (2), $\forall j< i, m(S_i, C_{j})=0$. If $[C]$ is not independent of  $[C_1], [C_2], \cdots, [C_{i-1}]$, then the annotation $a(C)$ of the cycle $C$ can be written as $a(C)=\sum_{j<i} \alpha_{j} a(C_{j})$,  where $\alpha_{j} \in \{0,1\}$ and at least one $\alpha_{j}=1, j<i$. Since $m(S_i, C_{i})=1$, we have $\sum_{j< i} \alpha_j m(S_i, C_{j})=1$. It follows that there exists at least one $C_{j}$, $j<i$, with $m(S_i, C_{j}) =1$, which contradicts with property (2). Therefore, $[C]$ is independent of  $[C_1], [C_2], \cdots, [C_{i-1}]$.
	\qed
\end{proof}
The following theorem guarantees that the above three conditions suffice for a minimal homology basis. Its proof is almost the same as the proof of~\cite[Theorem 1]{DivideConquer} (which draws upon the idea of \cite{depina}). %\TL{I omit the proof due to the page limit}%For completeness, we include the proof in Appendix \ref{appendix:main-thm}.
\begin{theorem} The set $\{C_1, C_2 , \cdots, C_g\}$ computed by maintaining properties (1), (2) and (3) is a minimal homology basis.
	\label{main-thm}
\end{theorem}

Taking advantage of the above theorem, we aim to compute a homology basis iteratively while maintaining  conditions (1), (2), and (3).
\subsubsection{Maintaining support vectors and computing shortest cycles.}
Now we describe the algorithm \Call{CycleBasis}{$\mathcal{G}$} (given in Algorithm \ref{basis}) that computes a minimal homology basis. In this algorithm, we first initialize each support vector $S_i$ so that only the $i$-th bit is set to 1. Then the main computation is done by calling the procedure \Call{ExtendBasis}{$1, g$}.  
\begin{algorithm}
	\caption{Computing a minimal Basis}\label{basis}
	\begin{algorithmic}[1] %??????
		\Procedure {CycleBasis}{$\mathcal{G}$}
		\For{$i\gets 1$ to $g$}
		\State Initialize $S_i\gets \{e_i\}$,  which means that the $i$-th bit of $S_{i}$ is 1 while others are 0
		\EndFor
		\State \Call{ExtendBasis}{$1, g$} to get a minimal homology basis $\{C_1, \cdots, C_g\}$
		\EndProcedure
	\end{algorithmic}
\end{algorithm}

Here the procedure \Call{ExtendBasis}{$i$, $k$} (Algorithm \ref{alg:extend}) is recursive which extends the current partial basis $\{C_1, \cdots, C_{i-1}\}$ by adding $k$ new cycles. It modifies a divide and conquer approach of \cite{DivideConquer} to maintain properties (1), (2), and (3). It calls a routine \Call{Update}{} to maintain orthogonality using annotations. For choosing the shortest cycle satisfying condition (3), it calls \Call{ShortestCycle}{$S_i$} in the base case ($k=1$)(See line 3 of Algorithm \ref{alg:extend}). We describe the recursion and the base case below.
\begin{algorithm}
	\caption{Extend the Basis by k elements} \label{alg:extend}
	\begin{algorithmic}[1] %??????
		\Procedure {ExtendBasis}{$i$, $k$}
		\If{$k=1$}
		\State Call \Call{ShortestCycle}{$S_{i}$} to compute the shortest cycle $C_{i}$ which is non-orthogonal to $S_{i}$
		\Else
		\State Call \Call{ExtendBasis}{$i, \lfloor k/2\rfloor$} to extend the homology basis by $\lfloor k/2\rfloor$ elements. After calling, $S_{i}, \text{...}, S_{i+\lfloor k/2 \rfloor-1}$ will be updated.
		\State Call \Call{Update}{$i$, $k$} to update the support vectors $\{S_{i+\lfloor k/2 \rfloor}, \text{...}, S_{i+k-1}\}$ using $\{S_{i}, \text{...}, S_{i+\lfloor k/2 \rfloor-1}\}$ and update the value $m(S_j, e)$ for $i+\lfloor k/2 \rfloor \leq j\leq i+k-1$ and every edge $e$.
		\State Call \Call{ExtendBasis}{$i+\lfloor k/2 \rfloor$, $\lceil k/2 \rceil$} to extend the cycle basis by $\lceil k/2 \rceil$ elements
		\EndIf
		\EndProcedure
	\end{algorithmic}
\end{algorithm}

\paragraph{Recursion.}
%\noindent \textbf{Recursion.~}
At the high level, the procedure \Call{ExtendBasis}{$i$, $k$} recurses on $k$ by first calling itself to obtain the next $\lfloor k/2 \rfloor$ cycles in the minimal homology basis in which the support vectors $S_{i}, S_{i+1}, \cdots, S_{i+\lfloor k/2 \rfloor-1}$ are updated. 
Then it calls the procedure \Call{Update}{$i$, $k$} to maintain the orthogonality property (2).  It uses the already updated support vectors $S_{i}, \cdots, S_{i+\lfloor k/2 \rfloor-1}$ to update $\{S_{i+\lfloor k/2 \rfloor}, \text{...}, S_{i+k-1}\}$ so that $m(S_l, C_j)=0, \forall j<i+\lfloor k/2 \rfloor, i+\lfloor k/2 \rfloor\leq l\leq i+k-1$. 
At last the procedure \Call{ExtendBasis}{$i$, $k$} calls itself \Call{ExtendBasis}{$i+\lfloor k/2 \rfloor$, $\lceil k/2 \rceil$} to extend the basis by $\lceil k/2 \rceil$ elements.

We describe \Call{Update}{$i$, $k$} and spare giving its pseudocode. Let $\{\hat S_{i+\lfloor k/2 \rfloor}, \text{...}, \hat S_{i+k-1}\}$ denote the desired output vectors after the update. To ensure the property (1) and (2), we will enforce that the vector $\hat S_{j}$ is of the form $\hat S_{j}=S_{j}+\sum_{t=1} ^{\lfloor k/2 \rfloor}\alpha_{jt}S_{i+t-1}$ where $i+\lfloor k/2 \rfloor\leq j\leq i+k-1$. We just need to determine the coefficients $\alpha_{j1}, \text{...}, \alpha_{j\lfloor k/2 \rfloor}$ so that $m( \hat S_{j}, C_{t} )=0$ where $i+\lfloor k/2 \rfloor\leq j\leq i+k-1$ and $i\leq t\leq i+\lfloor k/2 \rfloor-1$.
We will also compute $m(S_j, e)$ for $i+\lfloor k/2 \rfloor\leq j\leq i+k-1$ and every edge $e$ where $m(S_j, e)$ is defined as the standard inner product of $S_j$ and $a(e)$ under $\mathbb{Z}_2$, which is important later when we compute the shortest cycle orthogonal to a support vector $S$ in the procedure \Call{ShortestCycle}{$S$}. 

Now let 
\[
X=
\left( \begin{array}{c}
S_{i} \\
S_{i+1}\\
\vdots \\
S_{i+\lfloor k/2 \rfloor-1} \end{array} \right)
\cdot
\left(\begin{array}{cccc} 
a(C_{i})^T & a(C_{i+1})^T & \cdots & a(C_{i+\lfloor k/2 \rfloor-1})^T
\end{array} \right)\]
\[
Y=
\left( \begin{array}{c}
S_{i+\lfloor k/2 \rfloor} \\
S_{i+\lfloor k/2 \rfloor+1}\\
\vdots \\
S_{i+k-1} \end{array} \right)
\cdot
\left(\begin{array}{cccc} 
a(C_{i})^T & a(C_{i+1})^T & \cdots & a(C_{i+\lfloor k/2 \rfloor-1})^T
\end{array} \right) , \]
where recall that $g$-bit  vector $a(C)$ is the annotation of a cycle $C$. 
Let $A$ denote a $\lceil k/2 \rceil \times \lfloor k/2 \rfloor$ matrix where row $j$ contains the bit $\alpha_{j+i+\lfloor k/2 \rfloor-1,1}, \cdots, \alpha_{j+i+\lfloor k/2 \rfloor-1, \lfloor k/2 \rfloor}$. It is not difficult to see that $AX+Y=0$, and that $X$ is invertible, which means that $A=-YX^{-1}=YX^{-1}$ since the computations are under $\mathbb{Z}_2$.

The next step is to update the value $m(S_j, e)$ to $m(\hat S_j, e)$ for every edge $e$ in $\K$, and $i+\lfloor k/2 \rfloor\leq j\leq i+k-1$. Note that the coefficients $\alpha_{jt}$ are now known and the updated vectors are $\hat S_j=S_{j}+\sum_{t=1} ^{\lfloor k/2 \rfloor}\alpha_{jt}S_{i+t-1}$, $i+\lfloor k/2 \rfloor\leq j\leq i+k-1$. Thus for every edge $e$, $m(\hat S_j, e)=m(S_{j}+\sum_{t=1} ^{\lfloor k/2 \rfloor}\alpha_{jt}S_{i+t-1}, e)=m(S_{j}, e)+\sum_{t=1} ^{\lfloor k/2 \rfloor}\alpha_{jt}m(S_{i+t-1}, e)$, $i+\lfloor k/2 \rfloor\leq j\leq i+k-1$.
Let $n_1$ be the number of edges in $\K$ and $U$ be the $\lceil k/2 \rceil\times n_1$ matrix where its $(t, j)$ entry is  $m(\hat S_{i+\lfloor k/2 \rfloor+t-1}, e_j)$. Set $W=[A | I]$ where $I$ is the $\lceil k/2 \rceil \times \lceil k/2 \rceil$ identity matrix.
Let $Z$ be a $k\times n_1$ matrix whose $(s, t)$ entry is $m(S_{i+s-1}, e_t)$.
Thus we have $U=WZ$. Since the $\lceil k/2 \rceil\times k$ matrix $W$ and $k\times n_1$ matrix $Z$ are already known, the matrix $U$ can be computed in $O(nk^{\omega-1})$ time by chopping $Z$ to $n_1/k$ number of $k\times k$ submatrices and performing $O(n_1/k)$ matrix multiplications of two $O(k) \times O(k)$ size matrices. After that, $m(\hat S_j, e)$ can be easily retrieved from the matrix $U$ in constant time.  
%queried in constant time for each $e$. \yusu{?}

\paragraph{Base case for selecting a shortest cycle.}\label{selectcycles}
We now implement the procedure \Call{ShortestCycle}{$S_{i}$} for the base case to compute the shortest cycle $C_i$ non-orthogonal to $S_i$, i.e. the shortest cycle $C_i$ satisfying $m(S_i, C_i)=1$. 
%Instead of computing the annotation of each cycle $C(p, e)$ directly which increases the time complexity, w
We assign a label $l_p(u)$ to each vertex $u$ and $p$. Labeling has been used to solve many graph related problems previously~\cite{greedy, annotation, label}.
%for example, Erickson and Whittlesey~\cite{greedy}, Busaryev, et al.~\cite{annotation} and Mehlhorn and Michail~\cite{label} also used this technique in related problems.
 
%apply the label technique used in \cite{label} to determine whether a cycle $C$ is orthogonal to $S_i$ or not. 

Given a vertex $p$ and the shortest path tree $T_p$ rooted at $p$, let $\Pi_p(u)$ for any vertex $u\in vert(\K)$ denote the unique tree path in $T_p$ from $p$ to $u$, and  let $l_p(u)$ denote the value $m(S_i, \Pi_p(u))$. Let $w$ denote the parent of $u$ in tree $T_p$ and $e_{uw}$ denote the edge between $u$ and $w$. Then $l_p(u)=l_p(w) +m(S_i, e_{uw})$. Thus for a fixed $p \in vert(\K)$, we can traverse the tree $T_p$ from the root to the leaves and compute the label $l_p(u)$ for all vertices $u$ in $O(n)$ time as $m(S_i, e)$ for every edge is already precomputed earlier in the procedure \Call{Update}{} and can be queried in O(1) time. Thus the total time to compute labels $l_p(u)$ for all $p, u\in vert(\K)$ is $O(n^2)$. 

Now given a fixed vertex $p$ and the shortest path tree $T_p$, we consider every cycle $C(p,e)$, where $e=(u,v)$ is a non-tree edge. We partition the cycle into three parts: the tree path $\Pi_p(u)$, the tree path $\Pi_p(v)$ and the edge $e$. Thus $m(S_i, C(p,e))=m(S_i, \Pi_p(u))+m(S_i, \Pi_p(v)))+m(S_i, e)=l_p(u)+l_p(v)+m(S_i, e)$, which can be computed in $O(1)$ time as all labels are precomputed. Note that there are $O(n^2)$ cycles in the candidate set $\mathcal{G}$ to be computed. It results that in $O(n^2)$ total time, one can compute $m(S_i, C)$ for all cycles $C\in \mathcal{G}$ and find the smallest one.% in Line 13 of Algorithm \ref{sc}. 

%\begin{algorithm} 
%\caption{Compute the shortest cycle }\label{sc}
%\begin{algorithmic}[1]
%\Procedure{ShortestCycle}{$S_{i}$} 
%%\State Invariant: $S_i$ is a $g$-bit vector such that $m(x, C_j)=0$, $j=1, 2, \cdots, i-1$.
%%in the subspace orthogonal to ${a(C_1), a(C_2), \cdots, a(C_{i-1})}$, i.e. $S_i$ is a non-trivial solution to a set of linear equations: $m(x, C_j)=0$, $j=1, 2, \cdots, i-1$.
%	\State Compute $m(S_i, e_j), \forall e_j\in edge(\K)$
%	\ForAll{tree $T(p)$}
%	\State Compute label $l_{p}(u)\gets m(S_i, path(p, u))$ for very vertices $u\in vert(\K)$
%	\EndFor
%	\ForAll{cycle $C(p, e), e=(u, v)$ in the candidate set}
%		\If{$m(S_i, e)=1$}
%		\State $m(S_i, C(p, e)) \gets l_{p}(u)+l_{p}(v)+1$
%		\Else
%		\State $m(S_i, C(p, e)) \gets l_{p}(u)+l_{p}(v)$
%		\EndIf
%	\EndFor
%	\State Find $C(p, e)$ which is the shortest cycle in the candidate set such that $m(S_i, C(p, e)) =1$
%	\State \Return{$C(p, e)$}
%   	\EndProcedure
%\end{algorithmic}
%\end{algorithm}

\subsection{Correctness and time complexity}\label{sec:timeh1}
To prove the correctness of Algorithm \ref{basis}, it is crucial to guarantee that the support vectors $S_i$s and the cycles $C_i$s satisfy the desirable properties. First, the set of support vectors $\{S_1, S_2, \cdots, S_g\}$ is a basis of $\{0, 1\}^g$ because of the construction of $\hat S_{i}$s in the procedure \Call{Update}{}. The property that $\forall j<i$, $m(S_i, C_j)=0$ holds, because the procedure \Call{Update}{} ensures that $S_i$ is taken as a non-trivial solution to a set of linear equations $m(x, C_j)=0$, $1\leq j\leq i-1$, which always admits at least one solution. Similarly, for any $i\leq g$, there exists at least one cycle $C$ such that the equation $m(S_i, C)=1$ holds since both $S_1,\ldots, S_i$ and $C_1, \ldots, C_{i-1}$ at this point only form partial basis of a space with dimension $g$. In the base case, \Call{ShortestCycle}{} computes this cycle $C$ satisfying exactly this property. Then, Theorem~\ref{main-thm} ensures the correctness of the algorithm.

%For every phase $i$, we always add the shortest cycle $C_i$ to the basis satisfying $m(S_i, C_i)=1$, which means that the cycle $C_i$ is the shortest cycle where the corresponding homology class $[C_i]$ is independent of the linear combinations of $[C_1], [C_2], \cdots, [C_{i-1}]$. We iterate for all $i$ from 1 to $g$ and finally compute a minimal homology basis.
%All computation are over $\mathbb{Z}_2$. The procedure is presented in Algorithm \ref{highlevel}.
%\begin{algorithm}
%\caption{Computing a new basis element}\label{basecase}
%\begin{algorithmic}[1]
%	\Procedure{ShortestCycle}{$S_i$}
%	\State Invariant: $S_i$ is a $g$-bit vector in the subspace orthogonal to ${a(C_1), a(C_2), \cdots, a(C_{i-1})}$, i.e. $S_i$ is a non-trivial solution to a set of linear equations: $m(x, C_j)=0$, $j=1, 2, \cdots, i-1$.
%	\State Compute the shortest cycle $C$ such that $m(S_i, C)=1$.
%	\EndProcedure
%\end{algorithmic}
%\end{algorithm}

%correctness
%time complexity
%\subsection{Time complexity}\label{DivideTime}
The total running time of our algorithm is $O(n^2 g+n^{\omega})$ and the analysis is as follows. The time to annotate edges and construct the candidate set is $O(n^{\omega}+n^2\log n)=O(n^{\omega})$ from Proposition \ref{TimeAnnotation} and \ref{candidate}. When computing the basis, the time of the procedure \Call{CycleBasis}{} is dominated by the time of \Call{ExtendBasis}{}. For each $i\leq g$, the time complexity of \Call{ExtendBasis}{$i$,$k$} is bounded by the following recurrence:
	$$T(i, k)=
	\left\{ 
	\begin{matrix} \mbox{the time of }\Call{ShortestCycle}{S_{i}} & k=1 \\
	  2T(\cdot, k/2)+O(k^{\omega -1}n) & k>1  
	  \end{matrix} \right. $$
	  
Note that in the recursion, only the second parameter $k$ counts for the time complexity. %\yusu{check with my changes above and below on $T()$ (from $T(k)$ etc to $T(\cdot, k)$).}
Actually for each $i\leq g$, the time complexity of \Call{ShortestCycle}{$S_i$} in the base case is only $O(n^2)$ as we argued earlier, that is, $T(\cdot, 1)=O(n^2)$. Then the recurrence solves to $T(\cdot, k)=O(k(n^2)+k^{\omega-1}n)$. It follows that $T(1, g)=O(n^2g+g^{\omega-1}n)$. Combined with the time for computing annotations and constructing the candidate set, the time complexity is $O(n^2 g+n^{\omega})$.

%%%%%%%%%APPROXIMATION%%%%%%%%%%%
\section{An approximate minimal homology basis of $\mathsf{H}_1(\K)$}
\label{sec:approximation}

%In this section, we discuss the framework to compute a minimum basis, the divide-and-conquer technique, can also be used when computing an approximation of a minimum homology basis of $\mathsf{H}_1$ group. The bottleneck of computing the exact minimum homology basis is the procedure to compute the minimal  $1$-cycle $C$ such that $m(S_i, C)=1$. So one natural question is that whether the time complexity can be improved if we relax the constraint that the homology basis should be minimum and focus on an homology basis whose size is $t$-times the size of a minimum homology basis. Thus in this section, we discuss the situation when $t=2k-1$ where $k$ is a positive integer and a 2-approximate minimum homology basis. 

%\yusu{I noticed that notations like minimum homology basis, cycle basis etc are not properly defined before (which should be in either Section 2 or 3. }

In this section, we present an algorithm to compute an approximate minimal 1-dimensional homology basis, where the approximation is defined as follows. 
\begin{definition}[Approximate minimal homology basis]
Suppose $\mathcal{C}^*$ is a minimal homology basis for $\mathsf{H}_1(\K)$, and let $\ell^*_1 \le \ell^*_2\le \cdots \le \ell^*_g$ denote the sequence of sizes of cycles in $\mathcal{C^*}$ sorted in non-decreasing order. 
A set of $g$ cycles $\mathcal{C}'$ is a \emph{$c$-approximate minimal homology basis for $\mathsf{H}_1(\K)$} if (i) $\{ [C], C\in \mathcal{C}' \}$ form a basis for $\mathsf{H}_1(\K)$; and (ii) let $\ell_1, \ldots, \ell_g$ denote the sequence of sizes of cycles in $\mathcal{C}'$ in non-decreasing order, then for any $i\in [1,g]$, $\ell^*_i \le \ell_i \le c\cdot \ell^*_i$. 
\end{definition}

In what follows, we provide a 2-approximation algorithm running in $O(n^{\omega}\sqrt{n \log  n})$ time. 
At the high level, we first compute a set of candidate set $\mathcal{G}'$ of cycles that guarantees to contain a 2-approximate minimal homology basis. We then extract a 2-approximate basis from the candidate set $\mathcal{G}'$. 

First, we explain the construction of a candidate set of cycles. 
Recall that in Section \ref{sec:h1b}, we compute $O(n^2)$ candidate cycles, each of which has the form $C(p, e)$, formed by $e$ together with the two tree-paths from root $p$ to each of the endpoint of $e$ within the shortest path tree $T_p$.   
We now apply the algorithm by Kavitha et al. \cite{kavitha2007new} which can compute a smaller candidate set $\mathcal{G}'$ of $O(n \sqrt{n \log n})$ cycles which is guaranteed to contain a 2-approximate minimal \emph{cycle basis} (not homology basis) for graph $\K^{(1)}$ (i.e, 1-skeleton of the complex $\K$) in $O(n \sqrt{n}\log^{3/2} n)$ \emph{expected time}. 
Here, a cycle basis $\Gamma = \{\gamma_1, \ldots, \gamma_L \}$ of the graph $G = \K^{(1)}$ where $L=rank(\mathsf{Z}_1)$ is simply a set of cycles such that any other cycle from $G$ can be represented uniquely as a linear combination of cycles in ${\Gamma}$. 
A \emph{minimal cycle basis} is a cycle basis $\Gamma^*$ whose total weight $\sum_{\gamma \in \Gamma^*} \myS(\gamma)$ is smallest among all cycle basis. 
A cycle basis $\Gamma$ is a $c$-approximate minimal cycle basis if its total weight is at most $c$ times that of the minimal cycle basis, i.e, at most $c \cdot \sum_{\gamma \in \Gamma^*} \myS(\gamma)$.

Now let the size $\myS(\gamma)$ of a cycle be the total weight of all edges in $\gamma$. 
Then, it turns out that, $\mathcal{G}'$ not only contains a $2$-approximate minimal cycle basis w.r.t. this size, it also satisfies the following stronger property as proven in \cite{kavitha2007new}. 

\begin{proposition}[{\cite[Lemma 6.3]{kavitha2007new}}] \label{prop:2-app}
There exists a minimal cycle basis $\Gamma^* = \{\gamma^*_1, \ldots, \gamma^*_L\}$ such that, 
for any $1\leq i \leq L$, there is a subset $\Gamma_i \subseteq \mathcal{G'}$ of the computed candidate set $\mathcal{G}'$  so that (i) $\gamma^*_i = \sum_{\gamma \in \Gamma_{i}} \gamma$ and (ii) each cycle in $\Gamma_i$ has size at most $2\myS(\gamma^*_i)$.
\end{proposition}

%\begin{prop}\label{2-basis}
%The set $\mathcal{G'}$ contains $L$ linearly independent cycles $A_1, \cdots , A_L$ with $w(A_i ) \leq 2w(B_i )$,  for $i = 1, \cdots , L$.
%\end{prop}
%Now, let $\{C^{*}_1, C^{*}_2, \cdots, C^{*}_g\}$ denote a \emph{minimum homology basis}. 
Next, we prove that a candidate set $\mathcal{G'}$ satisfying conditions in Proposition \ref{prop:2-app} is guaranteed to also contain a 2-approximate minimal homology basis. We remark that if Proposition \ref{prop:2-app} does not hold, then the sole condition that $\mathcal{G}'$ contains a $c$-approximate minimal \emph{cycle basis} is \textbf{not} sufficient to guarantee that it also contains a $c$-approximate minimal \emph{homology basis} for any constant $c$. A counter-example is given at the end of this section. 

\begin{lemma}\label{lem:apprbasis}
Given a set $\mathcal{G'}$ of cycles satisfying Proposition \ref{prop:2-app}, there exists a minimal homology basis $\mathcal{C}^* = \{ C^*_1, \ldots, C^*_g \}$ such that $\mathcal{G}'$ contains $g$ cycles $A_1, \cdots , A_g$ with (i). $[A_1], \cdots, [A_g]$ form a homology basis, and (ii) $\myS(A_i ) \leq 2\myS(C^{*}_i )$,  for $i = 1, \cdots , g$.
\end{lemma}
\begin{proof}
Let $\Gamma^*$ be a minimal homology basis which satisfies Proposition \ref{prop:2-app}. It is known that it contains a minimal homology basis, which we set as $\mathcal{C}^* = \{ C^*_1, \ldots, C^*_g \}$. 
Now by Proposition \ref{prop:2-app}, for each $C^*_i$, there exists a subset $\Gamma_i \subseteq \mathcal{G'}$ such that $C^*_i = \sum_{\gamma \in \Gamma_i} \gamma$ and $\myS(\gamma) \le 2 \myS(C^*_i)$, $\forall \gamma \in \Gamma_i$. 
Assume w.l.o.g. that cycles in $\mathcal{C}^*$ are in non-decreasing order of their sizes. 
We now prove the lemma inductively. In particular, \\

{\sf Claim-A:} For any $k$, we show that there exists $A_1, \ldots, A_k\in \bigcup_{r \le k} \Gamma_r$ such that for each $i\in [1,k]$, (Cond-1) $\myS(A_i) \le 2 \myS(C^*_i)$; and (Cond-2) $[A_1], \ldots, [A_k]$ are independent.\\

The base case is straightforward: We can simply take $A_1$ as any cycle from $\Gamma_1$ that is not null-homologous (which must exist as $C^*_1 = \sum_{\gamma \in \Gamma_1} \gamma$ is not null-homologous). 

Now suppose the claim holds for $k$. Consider the case for $k+1$. 
By induction hypothesis, there exists $A_1, \ldots A_k\in \bigcup_{r \le k} \Gamma_r$ such that (Cond-1) and (Cond-2) hold. 
Now consider cycles in $\bigcup_{r \le k+1} \Gamma_r$. 
%If any cycle $\gamma \in \bigcup_{r \le k+1} \Gamma_r$ represents a homology class $[\gamma]$ independent of $[A_1], \ldots, [A_k]$, then simply set $A_{k+1} = \gamma$, and easy to see that both (Cond-1) and (Cond-2) holds for $A_{k+1}$. 
Let $\mathcal{H}_{k+1}$ denote the subgroup of $\mathsf{H}_1(\K)$ generated by the homology classes of all cycles in $\bigcup_{r \le k+1} \Gamma_r$. Note that $\mathcal{H}_{k+1}$ spans $\{[C^*_1], \ldots, [C^*_{k+1}]\}$, then the rank of $\mathcal{H}_{k+1}$ is at least $k+1$, which means there always exists a cycle $A_{k+1}\in \bigcup_{r \le k+1} \Gamma_r$ such that $[A_{k+1}]$ is independent of $[A_1], \ldots [A_k]$. By definition of $\bigcup_{r \le k+1} \Gamma_r$, there is an index $j\leq k+1$ such that $\myS(A_{k+1})\leq \myS(C^*_j)\leq \myS(C^*_{k+1})$ which satisfies both (Cond-1) and (Cond-2). Thus {\sf Claim-A} holds for $k+1$ as well.

%If such a $\gamma$ cannot be found, that means that for any $\gamma\in \bigcup_{r \le k+1} \Gamma_r$, $[\gamma] \in Span(\{ [A_1], \ldots, [A_k]\}$, which further implies that $C^*_{k+1}$ also satisfies that 
%\begin{align}
%[C^*_{k+1}] &= \sum_{i \le k} c_i [A_i] ~~ \text{where each coefficient~} c_i \in \{0, 1\}~\text{and at least one coefficient is}~1. 
%\label{eqn:Ck}
%\end{align}
%%$c_i \in \{0, 1\}$ and at least one coefficient is 1. 
%Let $\mathcal{H}_k$ denote the subgroup of $\mathsf{H}_1(\K)$ generated by the homology classes of all cycles in $\bigcup_{r \le k} \Gamma(C^*_r)$. Eqn (\ref{eqn:Ck}) implies that $[C^*_{k+1}] \in \mathcal{H}_k$. 
%
%On the other hand, the induction hypothesis implies that $\{[C^*_1], \ldots, [C^*_k]\} \subseteq \mathcal{H}_k$. 
%It means that there exists a cycle $\gamma$ which can be written as linear combination of cycles in $\bigcup_{r \le i} \Gamma(C^*_r)$ such that $[\gamma]$ is independent to $[A_1], \ldots, [A_k]$. 
%Assume that $\gamma = \gamma_{j_0} + \cdots + \gamma_{j_s}$ with each $\gamma_{j_i} \in \bigcup_{r \le i} \Gamma(C^*_r)$, meaning that $\myS(\gamma_{j_i}) \le 2 \myS(C^*_k)$ for any $i \in [1,s]$. 
%At least one $\gamma_{j_i}$ statisfies that $[\gamma_{j_i}]$ is independent of $\{[A_1], \ldots, [A_k]\}$ (as other wise, $[\gamma]$ has to be linearly dependent to $\{[A_1], \ldots, [A_k]\}$ as well). 
%We set $A_{k+1}$ to be this cycle $\gamma_{j_i}$ and both (Cond-1) and (Cond-2) hold. Thus {\sf Claim-A} holds for $k+1$ as well. 
The lemma then follows when $k=g$. \qed

%Given the optimal homology basis $\{ C^*_i, i\in [1,g] \}$, there always exists a minimal cycle basis, called it $\Gamma^*$, such that $\{ C^*_i, i\in [1, g]\} \subseteq \Gamma^*$ . 
%Thus for each $1\leq i\leq g$, we have $C^{*}_i =\sum_{C\in \mathcal{C}_i}C$ where $\mathcal{C}_i$ is a subset of $\mathcal{G'}$ and each cycle in $\mathcal{C}_i$ has size at most $2S(C^{*}_i)$. Let $\mathcal{H}_i$ denote the set of homology classes of cycles in $\mathcal{C}_i$, i.e. $\mathcal{H}_i=\cup_{C\in \mathcal{C}_i}[C]$. Thus for every $i$, the set $\cup_{r\leq i}\mathcal{H}_r$ spans $[C^{*}_1], [C^{*}_2], \cdots, [C^{*}_i]$. Suppose $j$ is the smallest number such that we cannot find an independent set $[A_1], \cdots, [A_j]$ with $w(A_i)\leq 2w(B_i)$, $1\leq i\leq j$. Then the set $\cup_{r\leq j-1}\mathcal{H}_r$ contains an independent set $[A_1], \cdots, [A_{j-1}]$ with $w(A_i)\leq 2w(B_i)$, $1\leq i\leq j-1$. However, since $\cup_{r\leq j}\mathcal{H}_r$ spans $[C^{*}_1], [C^{*}_2], \cdots, [C^{*}_j]$, it certainly has another cycle $A_j$ such that $[A_j]$ is independent with $[A_1], \cdots, [A_{j-1}]$. Thus $A_j\in \mathcal{C}_i$ with some $a\leq j$. Then we have $w(A_{j})\leq2w(C^{*}_a)\leq2w(C^{*}_j)$, a contradiction.
\end{proof}

So far we have proved that the new candidate set $\mathcal{G'}$ always contains a 2-approximate minimal homology basis. What remains is to describe how to compute such an approximate basis from the candidate set $\mathcal{G}'$. 
First, as in Algorithm \ref{alg}, we compute the annotation of all edges in $O(n^{\omega})$ time. Let $a(e)$ denote the annotation of an edge $e \in \K^{(1)}$ in the complex $\mathcal{K}$; recall that $a(e)$ is a $g$-bit vector with $g = rank(\mathsf{H}_1(\K))$.
Also recall that given a cycle $\gamma$, its annotation $a(\gamma) = \sum_{e\in \gamma} a(e)$ represents the homology class of this cycle, and two cycles are homologous if and only if they have the same annotation vectors. 

Now order the cycles in $\mathcal{G}' = \{\gamma_1, \ldots, \gamma_m\}$, where $m = |\mathcal{G}'| = O(n\sqrt {n \log n})$, in non-decreasing order of their sizes. 
We will compute the annotation of all cycles in $\mathcal{G}'$ and put them in the $g \times m$ matrix $M$, whose $i$-th column $M[i]$ represents the annotation vector for the cycle $\gamma_i$. 
Since $\mathcal{G}'$ contains a homology basis of $\mathsf{H}_1(\K)$ ( Lemma \ref{lem:apprbasis}), $rank(M) = g$. 

First, we explain how to compute annotation matrix $M$ efficiently. 
Let $edge(\K) = \{e_1, \ldots, e_L\}$ denote all edges from $\K$. 
Let $A$ denote the $L \times m$ matrix where $\gamma_i = \sum_{j\in [1,L]} A[i][j] e_j$; that is, non-zero entries of the $i$-th column $A[i]$ encode all edges in the cycle $\gamma_i$. 
Let $B$ denote the $g \times L$ matrix where the $i$-th column $B[i]$ encodes the annotation of edge $e_i$. 
It is easy to see that $M = A^T \cdot B^T$. 
Instead of computing the multiplication directly, we partition the matrix $A^T$ top-down into $m/L$ submatrices each of size at most $L\times L$. For each of this submatrix, its multiplication with $B^T$ can be done in $O(L^{\omega})$ matrix multiplication time. Thus the total time to compute the multiplication $M=A^T \cdot B^T$ takes $O(\frac{m}{L} L^{\omega}) = O(m n^{\omega-1})$ time as $L \le n$.  
In other words, we can compute the annotation matrix $M$ in $O(n^\omega \sqrt{n \log n})$ as $m = O(n \sqrt{n\log n})$.

We now compute a 2-approximate minimal homology basis from $\mathcal{G}'$. Here we use so-called \emph{earliest basis}. 
Specifically, in general, given a matrix $D$ with rank $r$, the set of column vectors $\{D[i_1], \cdots , D[{i_r}] \}$ is called an \textbf{earliest basis} for the vector space spanned by all columns in $D$ (or simply, for $D$), if the column indices $\{i_1,\ldots , i_r\}$ are the lexicographically smallest index set such that the corresponding columns of $D$ have full rank.

%\yusu{Why are there two Proposition 4.1? It seems that you have two different proposition environments mixed in the tex file. Please use only one.}
\begin{proposition}[\cite{annotation}]\label{prop:annotation}
\textit{Let $D$ be an $m\times n$ matrix of rank $r$ with entries over $\mathbb{Z}_2$ where $n \leq m$, then there is an $O(m n^{\omega-1})$ time algorithm to compute the earliest basis of $D$.}
\end{proposition}

Let $\{i_1, \ldots, i_g\}$ be the indices of columns in the earliest basis of $M$. 
This can be done in $O(m g^{\omega-1}) = O(n \sqrt{n\log n}\cdot g^{\omega-1})$ time by the above proposition as $m=O(n\sqrt {n \log n})$.  
The cycles corresponding to these columns form a homology basis by the properties of annotations \cite{annotation}. 

Finally, we note that the earliest basis of $M$ has the smallest (lexicographically) sequence of size sequence. Hence its total size is at most the size of the 2-approximate minimal homology basis $A_1, \ldots, A_g$ as specified in Lemma \ref{lem:apprbasis}. 
Hence putting everything together, we conclude with the following theorem. 

\begin{theorem}
The algorithm above computes a 2-approximate minimal homology basis of the 1-dimensional homology group $\mathsf{H}_1(\K)$ of a simplicial complex with non-negative weights in $O(n^{\omega}\sqrt{n \log  n})$ expected time.
\end{theorem}

\subsubsection{Remark.}
Since an approximate minimal homology basis still forms a basis for $\mathsf{H}_1(\K)$, it means that computing it is at least as hard as computing the rank of $\mathsf{H}_1(\K)$. Currently the best algorithm for the rank computation for general simplicial complex $\K$ is $O(n^\omega)$ (the matrix multiplication time). 
Hence the best we can expect for computing an approximate minimal homology basis is perhaps $O(n^\omega)$ (versus the $O(n^2 g+ n^\omega)$ time complexity of the exact algorithm from Section \ref{sec:h1b}). 
We remark that we can also develop an algorithm that computes a $(2k-1)$-approximate minimal homology basis in time $O(kn^{1+1/k} g ~\mathrm{polylog}~ n+ n^\omega)$, where $k\ge 1$ is an integer -- indeed, as the approximation factor reaches $\log n$, the time complexity becomes $O(n^\omega)$ (which is the best time known for rank computation). The framework of this algorithm follows closely from an approach by Kavitha et al. in \cite{kavitha2007new}, and we thus omit the details here. 

\subsubsection{A counter-example.}
Figure~\ref{exampleMCB} gives an example which shows that, without Proposition~\ref{prop:2-app}, it is not guaranteed that a candidate set containing a $c$-approximate minimal cycle basis includes a $2$-approximate minimal homology basis. 
Let the size of a 1-cycle in $\K$ shown in the figure to be the sum of all edges in the cycle. There is only one minimal cycle basis in this figure, namely $C_1, C_2, C_3$ and $C_4$, as shown in Figure~\ref{exampleMCBMCB}. 
The minimal homology basis of $\K$ should be $\{C_1, C_2, C_3\}$. 
However, consider the candidate set $\mathcal{G}$ which contains 4 cycles as shown in Figure~\ref{exampleMCBAppro}: $C_2, C_3, C_4$ and $C'_4=C_1+C_2+C_3$. It is easy to check that these 4 cycles in $\mathcal{G}$ form a 2-approximate minimal \emph{cycle} basis. 
However, the smallest homology basis contained in $\mathcal{G}$, namely $C_2, C_3, C'_4(= C_1+C_2+C_3)$ is not a 2-approximate minimal homology basis. 

We can make this example into a counter-example for any constant factor approximation, by adding more $C'_i$'s (triangles) to the sequence, each of which is larger than the previous one and is also filled in. In other words, the optimal homology basis remains $\{C_1, C_2, C_3\}$, while the smallest-size homology basis from the 2-approximate minimal cycle basis is $\{C_2, C_3, \sum_{i>1} C_i\}$.

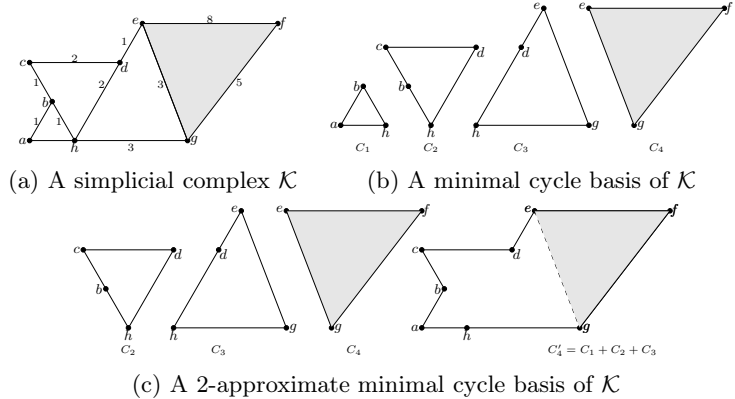
\begin{figure}[h]
\setlength{\intextsep}{5pt}%
\setlength{\columnsep}{8pt}%
\begin{center}
\begin{tabular}{cc}
\subfloat[A simplicial complex $\mathcal{K}$]{\label{exampleMCBComp}
\scalebox{.6}{
\begin{tikzpicture}[scale=0.5, auto,swap]
    \tikzstyle{node}=[font=\small,scale=1]
    \tikzstyle{edge} = [draw,thick,-]
    \tikzstyle{weight} = [font=\scriptsize]
    \tikzstyle{point}=[{draw, fill=black, circle, inner sep=1pt}]
    \foreach \pos/\name in {{(0,0)/aa}, {(1,1.73)/bb}, {(0,3.46)/cc}, {(4,3.46)/dd}, {(5,5.19)/ee}, {(11,5.19)/ff}, {(7,0)/gg}, {(2,0)/hh}}
    \node[point] (\name) at \pos {};
    \foreach \pos/\name in {{(-0.3,0)/$a$}, {(0.7,1.73)/$b$}, {(-.3,3.46)/$c$}, {(4.2,3.26)/$d$}, {(4.7,5.29)/$e$}, {(11.2,5.19)/$f$}, {(7.3,0)/$g$}, {(2,-.3)/$h$}}
    \node at \pos {\name};
    \node[weight] at (0.3, 0.85) {1};
    \node[weight] at (0.3, 2.58) {1};
    \node[weight] at (2, 3.66) {2};
    \node[weight] at (4.2, 4.39) {1};
    \node[weight] at (8, 5.39) {8};
    \node[weight] at (9.3, 2.6) {5};
    \node[weight] at (4.5, -0.3) {3};
    \node[weight] at (5.8, 2.5) {3};
    \node[weight] at (3.2, 2.5) {2};
    \node[weight] at (1.3, 0.85) {1};
    \filldraw[fill=black!10, draw=black] (5,5.19)--(11,5.19)--(7,0)--cycle;
     \draw (aa.center)--(bb.center)--(cc.center)--(dd.center)--(ee.center)--(gg.center)--(hh.center)--cycle;
     \draw (bb.center)--(hh.center);
     \draw (dd.center)--(hh.center);
     \end{tikzpicture}	
     }
    }
	&
\subfloat[A minimal cycle basis of $\mathcal{K}$]
{\label{exampleMCBMCB}
	\scalebox{.6}{
	\begin{tikzpicture}[scale=0.5, auto,swap]
		\tikzstyle{node}=[font=\small,scale=1]
       	 	\tikzstyle{edge} = [draw,thick,-]
        		\tikzstyle{weight} = [font=\scriptsize]
        		\tikzstyle{point}=[{draw, fill=black, circle, inner sep=1pt}]
        		\node[point] (a1) at (0,0){};
        		\node[point] (b1) at (1, 1.73){};
        \node[point] (h1) at (2, 0){};
        \foreach \pos/\name in {{(-0.3,0)/$a$}, {(0.7,1.73)/$b$},{(2,-.3)/$h$}}
    		\node (\name) at \pos {\name};
	\node[weight] at (1, -1) {$C_1$};
        \draw (a1.center) -- (b1.center) -- (h1.center) -- cycle;
        %\end{scope}
        \begin{scope}[xshift=2cm]
        \node[point] (b2) at (1, 1.73){};
        \node[point] (c2) at (0, 3.46){};
         \node[point] (d2) at (4, 3.46){};
         \node[point] (h2) at (2, 0){};
         \foreach \pos/\name in {{(0.7,1.73)/$b$}, {(-.3,3.46)/$c$}, {(4.2,3.26)/$d$}, {(2,-.3)/$h$}}
    \node (\name) at \pos {\name};
    \node[weight] at (2, -1) {$C_2$};
        \draw (b2.center) -- (c2.center) -- (d2.center) -- (h2.center) -- cycle;
        \end{scope}
         \begin{scope}[xshift=4cm]
         \node[point] (d3) at (4, 3.46){};
         \node[point] (e3) at (5, 5.19){};
         \node[point] (g3) at (7, 0){};
        \node[point] (h3) at (2, 0){};
        \foreach \pos/\name in {{(4.2,3.26)/$d$}, {(4.7,5.29)/$e$}, {(7.3,0)/$g$}, {(2,-.3)/$h$}}
    \node (\name) at \pos {\name};
    \node[weight] at (4, -1) {$C_3$};
        \draw (d3.center) -- (e3.center) -- (g3.center) -- (h3.center) -- cycle;
        \end{scope}
         \begin{scope}[xshift=6cm]
         \node[point] (e4) at (5, 5.19){};
        \node[point] (f4) at (11, 5.19){};
        \node[point] (g4) at (7, 0){};
        \foreach \pos/\name in {{(4.7,5.29)/$e$}, {(11.2,5.19)/$f$}, {(7.3,0)/$g$}}
    \node (\name) at \pos {\name};
    \node[weight] at (8, -1) {$C_4$};
         \filldraw[fill=black!10, draw=black] (e4.center)--(f4.center)--(g4.center)--cycle;
         \end{scope}
	\end{tikzpicture}
	}
}
\end{tabular}
\subfloat[A 2-approximate minimal cycle basis of $\mathcal{K}$]{\label{exampleMCBAppro}
	%\centering
	\scalebox{.6}{
	\begin{tikzpicture}[scale=0.5, auto,swap]
	\tikzstyle{node}=[font=\small,scale=1]
        \tikzstyle{edge} = [draw,thick,-]
        \tikzstyle{weight} = [font=\scriptsize]
        \tikzstyle{point}=[{draw, fill=black, circle, inner sep=1pt}]
        \node[point] (b2) at (1, 1.73){};
        \node[point] (c2) at (0, 3.46){};
         \node[point] (d2) at (4, 3.46){};
         \node[point] (h2) at (2, 0){};
         \foreach \pos/\name in {{(0.7,1.73)/$b$}, {(-.3,3.46)/$c$}, {(4.2,3.26)/$d$}, {(2,-.3)/$h$}}
        \node (\name) at \pos {\name};
        \node[weight] at (2, -1) {$C_2$};
        \draw (b2.center) -- (c2.center) -- (d2.center) -- (h2.center) -- cycle;
         \begin{scope}[xshift=2cm]
         \node[point] (d3) at (4, 3.46){};
         \node[point] (e3) at (5, 5.19){};
         \node[point] (g3) at (7, 0){};
        \node[point] (h3) at (2, 0){};
        \foreach \pos/\name in {{(4.2,3.26)/$d$}, {(4.7,5.29)/$e$}, {(7.3,0)/$g$}, {(2,-.3)/$h$}}
    \node (\name) at \pos {\name};
    \node[weight] at (4, -1) {$C_3$};
        \draw (d3.center) -- (e3.center) -- (g3.center) -- (h3.center) -- cycle;
        \end{scope}       
         \begin{scope}[xshift=4cm]
         \node[point] (e4) at (5, 5.19){};
        \node[point] (f4) at (11, 5.19){};
        \node[point] (g4) at (7, 0){};
        \foreach \pos/\name in {{(4.7,5.29)/$e$}, {(11.2,5.19)/$f$}, {(7.3,0)/$g$}}
    \node (\name) at \pos {\name};
    \node[weight] at (8, -1) {$C_4$};
         \filldraw[fill=black!10, draw=black] (e4.center)--(f4.center)--(g4.center)--cycle;
         \end{scope}         
         \begin{scope}[xshift=15cm]
         \node[point] (a1) at (0,0){};
         \node[point] (b1) at (1, 1.73){};
         \node[point] (c1) at (0, 3.46){};
         \node[point] (d1) at (4, 3.46){};
         \node[point] (e1) at (5, 5.19){};
         \node[point] (f1) at (11, 5.19){};
         \node[point] (g1) at (7, 0){};
         \node[point] (h1) at (2, 0){};
	\foreach \pos/\name in {{(-0.3,0)/$a$}, {(0.7,1.73)/$b$}, {(-.3,3.46)/$c$}, {(4.2,3.26)/$d$}, {(4.7,5.29)/$e$}, {(11.2,5.19)/$f$}, {(7.3,0)/$g$}, {(2,-.3)/$h$}}
        \node at \pos {\name};
        \foreach \pos/\name in {{(4.7,5.29)/$e$}, {(11.2,5.19)/$f$}, {(7.3,0)/$g$}}
    \node (\name) at \pos {\name};
    \node[weight] at (8, -1) {$C'_4=C_1+C_2+C_3$};
         \draw (a1.center)--(b1.center)--(c1.center)--(d1.center)--(e1.center)--(f1.center)--(g1.center)--(h1.center)--cycle;
         \draw[dashed] (e1.center)--(g1.center);
         \filldraw[fill=black!10] (e1.center)--(f1.center)--(g1.center);
         \end{scope}
	\end{tikzpicture}
	}
	}
	\end{center}

    \caption{An example where an approximate minimal cycle basis \textbf{does not} contain an approximate minimal homology basis.}
  \label{exampleMCB}
\end{figure}
%%%%%%%%%%%%SECTION5%%%%%%%%%%%%%
\section{Generalizing the size measure} 
\label{sec:generalization}

%Previously nearly all people used the geodesic distance between nodes as distance and the sum of edge weights as the size. A natural question is that whether there is any other way to measure a cycle. For example, \cite{chen2010measuring, guskov2001topological} use the idea of geodesic balls to measure a cycle. 
%Therefore in this section, we will give a generalization of the restriction of the distance function, which is called \textit{path-dominated distance} between two vertices. Then we will come up with a generalization of the size of a cycle in 1-dimension, under which we can always compute a minimum homology basis using the Algorithm \ref{basis}.
%Given a simplicial complex $K$ with weights defined on edges in the 1-skeleton $K^{(1)}$ of $K$, it is natural to use the induced shortest path distance in $K^{(1)}$ as a metric for vertices $V$ in $K$. 
The 1-skeleton $\K^{(1)}$ of the simplicial complex $\K$ is the set of vertices and edges in $\K$. 
If there are non-negative weights defined on edges in $\K^{(1)}$, it is natural to use the induced shortest path distance in $\K^{(1)}$ (viewed as a weighted graph) as a metric for vertices $V$ in $\K$. 
One can then measure the ``size'' of a cycle to be the sum of edge weights. 
Indeed, this is the distance and the size measure considered in Sections \ref{sec:h1} and \ref{sec:approximation}. 
In this section, we show that the algorithmic framework in Algorithm \ref{basis} can in fact be applied to a more general family of size measures. 
Specifically, first, we introduce what we call the \emph{path-dominated distance} between vertices of $\K$ (which is not necessarily a metric). 
Based on such distance function, we then define a family of ``size-functions'' under which measure we can always compute a minimal homology basis using Algorithm \ref{basis}. 
%We provide two examples of natural size-functions in Section \ref{subsec:examples}, one of which is the geomdesic ball-based measure used in \cite{chen2010measuring, guskov2001topological}. 
The shortest-path distance/size measure used in Section \ref{sec:h1}, and the geodesic ball-based measure proposed in \cite{chen2010measuring} are both special cases of our more general concepts. 
We also present another natural path-dominated distance function induced by a (potentially complex) map $F: vert(\mathcal{K}) \to Z$ defined on the vertex set  $vert(\mathcal{K})$ of $\K$ (where $Z$ is another metric space, say $\mathbb{R}^d$). 
As a result, we can use Algorithm \ref{basis} to compute the shortest 1-st homology basis of $\mathcal{K}$ induced by a map $F: vert(\mathcal{K}) \to Z$. 

\subsection{Path-dominated distance}%%%%%%%%%%%%%size
\label{subsec:distanceF}

Given a connected simplicial complex $\mathcal{K}$, suppose we are given a \emph{distance function} $d: vert(\mathcal{K}) \times vert(\mathcal{K}) \to \mathbb{R}^+\cup \{0\}$. 
We now introduce the following \emph{path-dominated distance function}. 

%In this part we give a restriction of the distance function from one vertex to another and name the functions that satisfy this restriction as path-dominated distance, which is the first step to define the size of a cycle. 
%Given a simplicial complex $\mathcal{K}$(suppose only 1 connected component), we use $n_0$ to denote the number of vertices, $n_1$ to denote the number of edges. The most natural way is to define the distance between two nodes as their geodesic distance which is the shortest path distance. Based on this idea, we now can define the path-dominated distance function $d(x,y)$ from vertex $x$ to $y$ where $x, y \in vert(\mathcal{K})$ as follows. 
\begin{definition}[Path-dominated distance]\label{def:distance} 
A function $d:vert(\mathcal{K})\times vert(\mathcal{K}) \to \mathbb{R}^+\cup\{0\}$ is a \textbf{path-dominated distance function (w.r.t. $\mathcal{K})$)} if 
\begin{itemize}
\item[(i)] $d(x,y) \ge 0$ and $d(x,x) = 0$ for any $x, y \in vert(\mathcal{K})$; 
\item[(ii)] given any two vertices $x, y \in vert(\mathcal{K})$, there exists a path $\pi^*$ connecting $x$ to $y$ in the 1-skeleton $\mathcal{K}^{(1)}$ such that $d(x,y) = \max_{u \in vert(\pi^*)} d(x,u)$. 
%for $d(x,u) \leq d(x,y))$ for any vertex $u\in vert(\pi^*)$.
%(1) $d(x,y) \geq 0$ and if $x=y$, $d(x,y)=0$. \\
%(2) for any two vertex $x, y\in V$, either $d(x,y) = \infty$ or there exists a path $\pi$ that for any vertices $u\in vert(\pi)$, $d(x,u) \leq d(x,y))$. 
\end{itemize}

%The path $\pi^*$  from condition (ii) above is referred to as a \textbf{\myshortestpath{}} between $x$ and $y$ in $\mathcal{K}$ w.r.t $d$. 
\end{definition} 

If edges in the 1-skeleton $\mathcal{K}^{(1)}$ have positive weights, then, it is easy to verify that the standard shortest path distance metric induced by $\mathcal{K}^{(1)}$ (viewed as a weighted graph) is path-dominated. %That is, setting $d(x,y)$ to be the shortest path distance  = d_{\mathcal{K}^{(1)}}(x,y)$, we have $d$ is path-dominated. 
However, note that a path-dominated distance may not be a metric. Indeed, we will shortly present a function-induced distance which is not symmetric. 

We now define ``shortest path'' in $\K^{(1)}$ induced by a path-dominated distance function. 
\begin{definition}[Path-dominated shortest path]\label{def:SPpath}
Given any $x, y \in vert(\K)$, a path $\pi^* = \langle u_0=x, u_1, \ldots, u_k = y$ connecting $x$ to $y$ via edges in $\K$ is a \textbf{\myshortestpath{} in $\K$} if for each $i \in [1, k]$, 
$d(x, u_i) = \max_{j \le i} d(x, u_j)$.  
\end{definition}

Note that this implies that any prefix of a \myshortestpath{} is also a \myshortestpath{}. 
The proof of the following statement is reasonably simple and can be found in Appendix \ref{appendix:claim:SPexists}. 

\begin{cl}\label{claim:SPexists}
A \myshortestpath{} always exists for any two vertices $x, y\in vert(\K)$. 
%(ii) Given a fixed vertex $p \in vert(\K)$, there exists a shortest path tree rooted at $p$ encoding a \myshortestpath{} from $p$ to every other vertex in $vert(\K)$. 
\end{cl}

\paragraph{Function-induced distance.}
\setlength{\intextsep}{5pt}%
\setlength{\columnsep}{8pt}%
\begin{wrapfigure}{r}{5cm}
\begin{tikzpicture}[scale=0.5]
	\tikzstyle{point}=[{draw, fill=white, circle, inner sep=0.5pt}]
	\node[left, font=\tiny] at (0,0) {$x$};
	\node[right, font=\tiny] at (0,-3) {$y$};
	\node[point] (x) at (0,0) {};
	\node[point] (y) at (0,-3) {};
	\node[left, font=\tiny] at (-.3, -.7) {$\pi_1$};
	\node[right, font=\tiny] at (.4, -1.5) {$\pi_2$};
	\draw (x.center) to [out=180,in=100] (y.center);
	\draw (x.center) to [out=300,in=20] (y.center);
	\begin{scope}[xshift=50, yshift=-30]
	\draw[->] (-.5,0) [in =150, out = 30] to (.5,0);
	\node[above, font=\tiny] at (0,0) {$F$};
	\end{scope}
	\begin{scope}[xshift=100, yshift=-50,yslant=0,xslant=1]
    %the rectangular surface onto which the clusters are located
    	\filldraw[black!10,very thick] (-1,-1) rectangle (5,2);
	\node[right, font=\tiny] at (-1,-.8) {$Z=\reals^2$};
	\draw [color=red]plot[smooth] coordinates {(0,0) (3.5, 1) (3,0)};
	\draw [color=blue]plot[smooth,tension=.9] coordinates {(0,0) (.3,-.1)(.5, 0.1) (.7,-.2)(1, .3)(1.3, -.3)(1.6, .2) (1.8,-.3)(2.1, .2)(2.4, -.3)(2.7, 0.2) (3, 0)};
	\node[left, font=\tiny] at (0,0) {$F(x)$};
	\node[right, font=\tiny] at (3,0) {$F(y)$};
	\node[point] (x) at (0,0) {};
	\node[point] (y) at (3,0) {};
	\node[left, font=\tiny] at (2, .7) {$F(\pi_1)$};
	\node[below, font=\tiny] at (2.4, -.3) {$F(\pi_2)$};
	\end{scope}
\end{tikzpicture}
\caption{The left is the original simplicial complex. There are two paths, $\pi_1$ and $\pi_2$, connecting vertices $x$ and $y$. The right figure is their image under the map $F$ with $Z = \reals^2$ (i.e, $d_Z (\cdot, \cdot) = \| \cdot - \cdot \|$). The path $\pi_2$ is a \myshortestpath{} from $F(x)$ to $F(y)$.}
\vspace{-30pt}
\label{fig:func}
\end{wrapfigure}
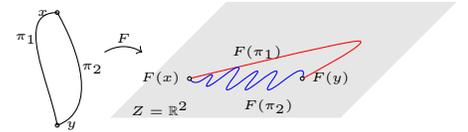

Very often, the domain $\K$ may come with additional data modeled by a function $F:vert(\mathcal{K})\to Z$ defined on vertices of $\mathcal{K}$, where the co-domain $(Z, d_Z)$ is a metric space. For example, imagine that $\K$ represents the triangulation of a region on earth, and at each vertex, we have collected $d$ sensor measurements (e.g. temperature, wind speed, sun-light strength, etc), which can be modeled by a function $F: vert(\K) \to \reals^d$. 
It is then natural to define a distance, as well as a size measure later, that depends on this function $F$. 
We introduce the following \emph{function-induced distance} $d_F: vert(\K) \times vert(\K) \to \reals$: 
\begin{definition}
Given any function $F: vert(\K)\to Z$, we define the \textbf{$F$-induced distance} $d_F(x, y)$ as follows: 
\begin{align}\label{eqn:funcD}
d_F(x, y) &= \min_{path~ \pi(x, y) \subseteq \K^{(1)}} \max_{u \in \pi(x, y)} d_Z (F(x), F(u))  ,
\end{align}
where the minimum ranges over all path $\pi(x,y)$ from $K^{(1)}$ connecting $x$ to $y$. 
\end{definition}

Intuitively, given a path $\pi$ from $x$ to $y$, $\max_{u \in \pi} d_Z(F(x), F(u))$ measures the maximum distance in terms of the function value $F$ between the starting point $x$ to any point in the path $\pi$, i.e, the maximum function distortion from $x$ to $\pi$. 
$d_F(x,y)$ is the smallest function distortion (w.r.t. $x$) needed to connect from $x$ to $y$. 
For example, in Figure \ref{fig:func}, the path $\pi_2$ is a \myshortestpath{} from $x$ to $y$, as its image $F(\pi_2)$ has a smaller maximum distance (in terms of $d_Z = \| \cdot \|$) than the image of $F(\pi_1)$. 
By the definition of function-induced distance, we have: 

\begin{cl}
Given $F: vert(\K)\to Z$, the $F$-induced distance $d_F$ is path-dominated. 
\end{cl}

\subsection{Size-measure for 1-cycles}%%%%%%%%%%%%%size
\label{subsec:size}

Previously, the most popular way to measure the ``size'' of a 1-cycle is the sum of weights of edges in the cycle. 
Another natural measure formulated by Chen and Freedman \cite{chen2010measuring} uses the minimum radius of any metric ball (centered at some vertex in $\K$) containing a cycle as its size. Intuitively, given a homology class, a smallest cycle of this class under this radius-measure corresponds to a cycle which is most ``localized'' (contained within a smallest possible metric ball). 
Using the shortest-path metric induced by weights on edges in $\K$, Chen and Freedman showed that a minimal homology basis under this radius-measure can be computed in polynomial time for any fixed-dimensional homology group. 
In what follows, we introduce a family size measures, which we refer to as \emph{\mysizeF{}s}, which generalize the radius-measure of Chen and Freedman as well as the general sum-of-weights measure. 
We then show that the algorithm from Section \ref{sec:h1} can be used to compute a minimal homology basis for $\mathsf{H}_1(\K)$ w.r.t. such \mysizeF{}s. 

We use the concept of \textit{edge-short} cycles introduced in e.g. \cite{gleiss2001short}, whose origin traces back to \cite{horton1987polynomial}. 
\begin{definition}
\label{def:edgeshort}
A 1-cycle $C$ in a complex $\mathcal{K}$ is called \textbf{edge-short}, if $\mathcal{K}$ contains a vertex $w$, an edge $e = (u, v)$, a shortest path from $w$ to $u$ and a shortest path from $w$ to $v$ such that $C$ is
the edge disjoint union of $e$ and the two paths.
\end{definition}

In our case, instead of using the shortest path metric induced by weights on the 1-skeleton $\K^{(1)}$ of $\K$, we use any path-dominated distance function $d$, and the ``shortest paths'' in the above definition will be replace by \myshortestpath{}s in $\K$ w.r.t. $d$. To emphasize the dependency on the path-dominated distance function $d$, we say that a cycle $C$ is \emph{edge-short w.r.t. $d$} if conditions in Definition \ref{def:edgeshort} holds w.r.t. \myshortestpath{}s w.r.t. $d$. 

%\yusu{The size $S$ may be confusing as we use $S_i$ as supporting vectors before. Probably we shoudl change $S$ to say $\mu$ if it is not yet used.}
\begin{definition}[\MysizeF{}]
Suppose $vert({\K})$ is equipped with a path-dominated distance function $d$. 
Let $\mathsf{Z}_1(\K)$ represent the 1-dimensional cycle group of $\K$. A function $\myS: \mathsf{Z}_1(\K)\to \mathbb{R}$ is a \textbf{\mysizeF{} (w.r.t. $d$)} if under this function, there exists a minimal homology basis for $\mathsf{H}_1(\K)$ in which all cycles are edge-short w.r.t. the path-dominated distance $d$. 

We may omit the reference to the path-dominated distance $d$ when its choice is fixed or clear. 
\end{definition}

We now prove that if a function is a \mysizeF, the Algorithm \ref{basis} can be used to compute a minimal homology basis. First, observe the following, which is implied by Theorem \ref{main-thm}. %\ref{lm:MinHomBasis}. 
%\yusu{Tianqi, please give a line where this claim follows from.}

\begin{cl} \label{cl:candidate}
If the candidate set $\mathcal{G}$ contains a minimal homology basis, then the framework Algorithm \ref{basis} will compute a minimal homology basis from the candidate set.
\end{cl}

What remains is to show how to compute a candidate set containing a minimal homology basis. 
For simplicity, from now we fix a path-dominated distance function $d$, and simply use \emph{shortest paths} to refer to the \emph{\myshortestpath{}s w.r.t. $d$}. 
We assume that the shortest paths are unique -- In Appendix \ref{appendix:ObtainLex}, we describe how to guarantee this uniqueness condition (by assigning certain order to the shortest paths), and show that the shortest path tree $T_p$ encoding all unique \myshortestpath{}s to any root $p\in vert(\K)$ can be computed in $O(n\log n)$ time (with $n = |\K^{(1)}|$) by the standard approach.  

We now construct a candidate set $\mathcal{G}$ in the same manner as in Section \ref{sec:h1}. 
First for every vertex $p$, we build a candidate set $\mathcal{G}_p$. Let $\Pi_p(u, v)$ denote the unique tree path between two vertices $u$ and $v$. For every nontree edge $e=(u, v)$, $C(p, e)=e\circ \Pi_p(u, v)$ is a cycle. We add all such $C(p, e)$ into $\mathcal{G}_p$, i.e. $\mathcal{G}_p=\cup_{e\in E\setminus edge(T_p)}C(p, e)$. Then take the union of all such candidate sets, $\mathcal{G}$ can be constructed as $\mathcal{G}=\cup_{p \in vert(\mathcal{K})}\mathcal{G}_p$.

%Now we present that under any size function, the candidate set we have always contains a minimal homology basis. The following lemma is needed for the proof.
\begin{lemma}\label{CanEdge}
The candidate set $\mathcal{G}$ contains a minimal homology basis when the size of a cycle is measured by a \mysizeF{} w.r.t. some path-dominated distance function $d$. 
%contains all edge-short cycles if the function to measure size of cycles is a \mysizeF{}. 
\end{lemma}
\begin{proof}
By results from Appendix \ref{appendix:ObtainLex}, we can assume that there is only a unique \myshortestpath{} between any two vertices $u, v\in vert(\K)$, which we denote as $SP(u, v)$. 
Now take any edge-short cycle $C$. As it is edge-short, we can find a vertex $w$ and an edge $e=(u, v)$ such that the cycle $C$ is the disjoint union of $SP(w, u)$, $SP(w, v)$ and $e$. On the other hand, the unique shortest paths $SP(w, u)$ and $SP(w, v)$ are in the shortest path tree $T_w$. This means that $e \notin T_w$. Hence the cycle $C$ is a candidate cycle $C(w,e)$ from the set $\mathcal{G}_w$. 
It then follows that the collection $\mathcal{G}$ contains all edge-short cycles. 
The lemma then follows from the definition of \mysizeF{}s and Lemma \ref{cl:candidate}. \qed
%We claim that the edge $e$ is the only edge in the cycle $z$ which is an nontree edge of tree $T_w$. Suppose not, then one of $SP(w, u)$ and $SP(w, v)$ is the union of the other and the edge $e$. Without loss of generality, we say the $SP(w, u)$ is $SP(w, v)$ together with $e$. Then the disjunctive union of $SP(w, u)$, $SP(w, v)$ and $e$ will be empty which is not a cycle, a contradiction. The fact that the edge $e$ is the only nontree edge in $z$ turns out that the cycle $z$ has the form $C(w, e)$. Therefore the cycle $z$ is in the candidate set.
\end{proof}

Now that we have a candidate set that contains a minimal homology basis, we can apply the divide and conquer algorithm (Algorithm \ref{basis}) from Section \ref{sec:h1}, and by Claim \ref{cl:candidate}, this will output a minimal homology basis. We conclude with the following main result. 

\begin{theorem}\label{thm:general}
Suppose sizes of 1-cycles are measured by a \mysizeF{} w.r.t. a path-dominated distance function $d$. 
Then, we can compute a minimal homology basis for $\mathsf{H}_1(\K)$  in $O(n^{\omega}+n^2g)$ time, where $n$ is the size of $2$-skeleton of $\K$ and $g$ is $rank(\mathsf{H}_1(\K))$. 
\end{theorem}

\subsection{Examples of \mysizeF{}s}
\label{subsec:examples}

\paragraph{Sum-of-weights size function.}
As mentioned earlier in Section \ref{subsec:distanceF}, given a weight function $w: edge(\K) \to \reals^+$, the shortest path distance $d_{\K}$ induced by the 1-skeleton $\K^{(1)}$ (viewed as a weighted graph) is a path-dominated function. 
%Also, given a function $F: vert(\K) \to Z$ defined on $\K$, the $F$-induced distance $d_F$ as introduced in Section \ref{subsec:distanceF} is also a path-dominated function. 
Now given weights $w: edge(\K) \to \reals^+$, the size measure $\myS_w: \mathsf{Z}_1(\K) \to \reals^+$ assigning $\myS_w(C) = \sum_{e\in C} w(e)$ is a \mysizeF{} w.r.t. the shortest path distance function $d_{\K}$. 
Hence we can obtain the main result of Section \ref{sec:h1} by applying Theorem \ref{thm:general} to the \mysizeF{} $\myS_w$. 

\paragraph{Radius-size function.}
Alternatively, we now consider the radius-based size function used e.g. in \cite{chen2010measuring,guskov2001topological, wood2004removing}. 
More specifically, suppose we are given a simplicial complex $\mathcal{K}$, and a path-dominated distance function $d$ (which may not be a metric) on $vert(\K)$. 
Define the ball $B_p^r$ centered at $p$ of radius $r$ to be $B_p^r=\{\sigma\in \mathcal{K}\text{ : }\forall x\in vert(\sigma), d(p, x)\leq r\}$. We can then define \emph{radius-size function $\myS_R: \mathsf{Z}_1(\K) \to \reals^+$} such that $\myS_R(C)$ of a 1-cycle $C$ is the smallest $r$ such that $C \subseteq B^r_p$ for some $p \in vert(\K)$. 
%(Note that this function can in fact be used to measure the size of any $d$-cycle, as e.g, in \cite{chen2010measuring}.)

% the Then we define the size function $\mu(C)$ of a cycle $C$ as the smallest ball centered at a vertex that contains $C$, i.e. $\mu(C)=min_{r}\{r\text{ : }\exists p\in vert(\mathcal{K})\text{ such that }C\subset B_p^r\}$. 

\begin{prop} \label{prop:radius}
%The candidate set always contains a minimum homology basis in which all cycles are edge-short.
$\myS_R$ is a \mysizeF{} w.r.t. any path-dominated distance function $d$. 
\end{prop}
\begin{proof} 
We need to prove that there exists a minimal homology basis where each cycle inside is edge-short. 
Assume this is not the case. 
Then given any minimal homology basis $\mathcal{B}$, there exists a cycle $C$ which is not edge-short. 
Suppose cycles $\mathcal{B} = \{C_1, \ldots, C_g\}$ are sorted in nondecreasing order of their radius-size, and $C_i$ is the first cycle in $\mathcal{B}$ that is not edge-short. %; that is, $C_1, \cdots, C_{i-1}$ are all edge-short. 
Let $B_p^r$ be the smallest ball containing $C_i$ with $p \in vert(\K)$; that is, $\myS_R(C_i) = r$. 
%Note that by definition of $B_p^r$, for any vertex $u$ in the ball $B_p^r$, the distance $d(p, u)\leq r$. 
Let $T_p$ denote the shortest path tree rooted at $p$, and $Q$ denote the set of edges in $C_i$ which are not in $T_p$. 
Note that $Q$ cannot be empty; otherwise, $C_i$ cannot be a cycle as all edges in it are tree edges. 
For every edge $e = (u,v)$ in $Q$, we can construct a cycle $C(p, e)$ as $SP(p,u)+SP(p,v)+e$, where $SP(x,y)$ denote the tree path in $T_p$ from $x$ to $y$. 
It is easy to see that for each such cycle $C(p, e)$ with $e\in Q$, its radius-size $\myS_R(C(p,e)) \le r$ as it is completely contained within $B^r_p$. 
Note that $C_i$ can be represented as the sum of all such $C(p, e)$, i.e. $C_i=\sum_{e\in Q} C(p, e)$. 
This is because that $C_i = \sum_{e \in C_i} C(p,e)$. However, for an edge $e\in C_i \cap T_p$, $C(p,e)$ is the empty set. Hence only edges from $C_i \setminus T_p (= Q)$ contribute to this sum. 
%Now we claim that the cycle $C_i$ can be represented as the sum of all such $C(p, e)$, i.e. $C_i=\sum_{e\in Q} C(p, e)$. This is because the set of cycle $\{ C(p, e') \mid e' \notin T_p\}$ forms a cycle basis for the graph $\K^{(1)}$. Thus every cycle can be written as a unique linear combination of these cycles that is decided by the existence of the nontree edges in $C_i$. 

Now consider the set of cycles $\mathcal{Q} = \{ C(p,e) \mid e\in Q\}$. 
As $C_i$ is in a minimal homology basis $\mathcal{B}$, its homology class $[C_i]$ is independent of those generated by cycles in $\mathcal{B} \setminus \{C_i\}$. Hence there exists at least a cycle $C'\in \mathcal{Q}$ such that $[C']$ is independent of the homology class of all cycles in $\mathcal{B} \setminus \{C_i\}$. Now let $\mathcal{B'}=\mathcal{B}\cup \{C'\} \setminus \{C_i\}$ which is also a homology basis.
%This means that we can replace $C_i$ by $C'$ in $B$ and obtain another homology basis $B'$. 
Recall that any cycle in $\mathcal{Q}$ has radius-size at most $r$. We have two cases:  
(i) If $\myS_R(C') < r  (=\myS_R(C_i)$, then $\mathcal{B'}$ has a smaller size sequence than $\mathcal{B}$, and thus $\mathcal{B}$ cannot be a minimal homology basis. Thus we have a contradiction, meaning that all cycles in $\mathcal{B}$ must be edge-short. 
(ii) If $\myS_R(C') = r$, then $\mathcal{B'}$ is also a minimal homology basis. 
If $B'$ contains only edge-short cycles, then we are done. 
If not, then we identify the next cycle that is not edge-short $C_j$, and it is necessary that $j > i$. We then repeat the above argument with $C_j$. 
In the end, either we find a contradiction, meaning that the edge-short cycle cannot exist in the basis we are inspecting, or we manage to replace all non-edge-short cycles to be edge-short ones of equal size, and maintain a homology basis. In the latter case, we construct a minimal homology basis with only edge-short cycles. 
In either case, the proposition follows. \qed
\end{proof}

It then follows from the above proposition that Algorithm \ref{basis} computes a minimal homology basis under the radius-size function w.r.t. any path-dominated distances in time $O(n^{\omega}+n^2g)$. 
In particular, combining with the two path-dominated distance functions examples we have: 

\vspace*{0.07in}\textbf{Example 1: } $d = d_\K$, the shortest path distance induced by the weighted graph $\K^{(1)}$. Under this path-dominated distance, the minimal homology basis problem under the radius-size function w.r.t. $d_\K$ is exactly the 1-dimensional case of the problem studied in \cite{chen2010measuring}. An $O(n^4 g)$ time algorithm was presented to solve this problem in any dimension in~\cite{chen2010measuring}. However, by Theorem \ref{thm:general}, we can compute a minimal homology basis of in $O(n^{\omega}+n^2g)$ time, which is a significant improvement when focusing on $\mathsf{H}_1$ group. \\

\textbf{Example 2: } Given a function $F: vert(\K) \to Z$ defined on $\K$, recall that the $F$-induced distance $d_F$ as introduced in Section \ref{subsec:distanceF} is a path-dominated function. 
Now set $d = d_F$. 
Intuitively, the radius-size function $\myS_R(C)$ w.r.t. $d_F$ measures the radius of the smallest metric ball in the co-domain $Z$ that contains the image $F(C)$ of the cycle $C$ under map $F$. That is, $\myS_R(C)$ measures the ``size'' of $C$ w.r.t. the variation in the function $F$. Hence we also refer to the radius-size function w.r.t. $d_F$ as the \emph{$F$-induced radius-size function}. 
We believe that the $F$-induced distance function and $F$-induced radius-size function are useful objects of independent interests. 
The minimal homology basis of $\K$ under such a $F$-induced radius-size function can also be computed in $O(n^{\omega}+n^2g)$ time.

\section{Conclusions}
In this paper we have given improved algorithms for computing a minimal homology basis for 1-dimensional homology group of a simplicial complex. What about higher dimensional homology?
For high dimensions, it is known from \cite{chen2010hardness} that computing a minimum homology basis under volume measure is NP-hard. But it follows from \cite{chen2010measuring} that one can extend the radius-size measure (See Section \ref{subsec:examples}) to high dimensions under which an algorithm to compute a minimum homology basis in polynomial time exists. It runs in time $O(gn^4)$ where $g$ is the rank of $d$-dimensional homology group $\mathsf{H}_d$. We can improve this algorithm, using persistence algorithm \cite{edelsbrunner2008persistent} as well as annotations for $d$-simplices \cite{annotation}, so that the time complexity improves to $O(n^{\omega+1})$ which is better when $g=\Theta(n)$. The details are presented in the Appendix A.

\section*{Acknowledgements}
This works is partially supported by National Science Foundation (NSF) under grants CCF-1526513, 1740761 and 1733798.
\bibliographystyle{splncs03}
\bibliography{CameraMCB}
\newpage
\appendix
%%%%%%%%%%%%%%%%Hd%%%%%%%%%
\section{Computing a minimal homology basis for $\mathsf{H}_d(\K)$}
\noindent Let $\mathcal{K}$ be a simplicial complex with $n$ simplices and let $g$ be the $d$-dimensional Betti number, i.e. $g=rank(\mathsf{H}_d(\K))$. The discrete geodesic distance $d_p: vert(\mathcal{K})  \to \mathbb{R}$ from a vertex $p$ is given by $q\mapsto dist(p,q)$ where $dist(p,q)$ is the length of the shortest path from $p$ to $q$. Extending this definition to general simplices, we have $\forall \sigma \in \mathcal{K}, d_p(\sigma)=max_{q\in vert(\sigma)}d_p(q)$. Then the geodesic ball $B_p^r$ of radius $r$ centered at $p$ is defined as $B_p^r=\{\sigma \in \mathcal{K}: d_p(\sigma)\leq r\}$. Clearly, $B_p^r\subseteq \mathcal{K}$, and it is a subcomplex of $\mathcal{K}$. This is because for all faces $\sigma'$ of $\sigma$, $d_p(\sigma') \leq d_p(\sigma)$, which implies that all faces of a simplex in $B_p^r$ are also in $B_p^r$.

The size of a cycle $C$ is defined as $\mu(C)=min\{r : \exists p\in vert(K), s.t. ~C\subset B_p^r\}$ \cite{chen2008quantifying}. In words, it is the radius of the smallest ball centered at some vertex $p$ of $C$, which contains $C$. The definition of a minimal homology basis becomes:
\begin{definition}\label{def:MCBd}
Given a simplicial complex $\K$,  a set of cycles $\{C_{1}, C_{2}, \cdots, C_{g}\}$ with  $g=rank(\mathsf{H}_d(\K))$ is a $d$-dimensional minimal homology basis if (1) the homology classes $\{[C_{1}], [C_{2}], \cdots, [C_{g}]\}$ constitute a homology basis and (2) the sizes $\{\mu(C_{1}), \mu(C_{2}), \text{. . .}, \mu(C_{g})\}$ are lexicographically smallest among all such bases.
\end{definition}

\subsection{Algorithm}
In this section, we describe an algorithm to compute a minimal $d$-dimensional homology basis where $d\geq 1$. There are two steps in the algorithm: First computing a candidate set which contains a minimal homology basis and then computing a minimal homology basis from the candidate set. All computations are over $\mathbb{Z}_2$.
%We use annotation to denote and distinguish each cycle. Hence the first step is to annotate all $d$-simplices. 

%\begin{algorithm}
%\caption{Minimal Homology Basis}\label{MHBd}
% \begin{algorithmic}[1] 
% 	%\State Annotate $d$-simplices in $\mathcal{K}$\cite{annotation}
%	\State Compute a candidate set which contains a minimal homology basis
%	\State Compute a minimal homology basis from the candidate set.
%\end{algorithmic}
%\end{algorithm}
\subsubsection{Computing a candidate set.}
We now describe how to compute a candidate set of cycles including a minimal homology basis. We apply the persistent homology algorithm to generate the candidate set $\mathcal{C}(p)$ for a vertex $p$ with the following filtration: Simplices are sequenced in non-decreasing order of geodesic distances $d_p(\cdot)$ while placing a simplex before all its cofaces that have the same geodesic distance.
% Note that for any simplex $\sigma$, any proper face $\sigma'$ of $\sigma$ has exact one of the following two properties: (1) $d_p(\sigma')<d_p(\sigma)$. (2) $d_p(\sigma')=d_p(\sigma)$ and the dimension of $\sigma'$ is smaller than $\sigma$. Hence the order of simplices is unique.
We focus on the essential homology classes computed by persistent algorithm. There are $g$ of them. For each essential homology class $h$, we denote its birth time as $r_p(h)$.
%We use the algorithm in \ref{chen2010measuring} to capture the cycle $C_h$ generated by a creator $\sigma$ of class $h$ and add the cycle $C_h$ into the candidate set. 
For any vertex $p$, the number of candidate cycles in $\mathcal{C}(p)$ is $g$. Thus, the number of cycles of the candidate set $\mathcal{C}$ is $O(g n)$. %The following claim follows and the proof is in the full version of the paper. %\TL{Professors, I omit the proof of this lemma due to the page limit}\\
\begin{cl}
The candidate set $\mathcal{C}$ includes a minimal homology basis.
\end{cl}
\begin{proof}
Suppose not. Let $\mathcal{B}$ be any minimal homology basis and the elements in $\mathcal{B}$ are sorted in nondecreasing order of their sizes. Let class $C_i$ be the first member in $\mathcal{B}$ which is not in the candidate set and let $p$ be the vertex such that $C_i\subset B_p^{\mu(C_i)}$ where $\mu(C_i)$ is the size of the cycle $C_i$. 
First we claim that there exists a $d$-simplex $\sigma$ such that $d_p(\sigma)=\mu(C_i)$ and $\sigma$ is a creator of $C_i$. If not, there is another cycle $C'$ such that $[C']=[C_i]$ and $\mu(C')<\mu(C_i)$. Note that the cycles generated by creators in $B_p^{\mu(C_i)}$ form a homology basis of $B_p^{\mu(C_i)}$. 
We prove that the geodesic ball $B_p^{\mu(C_i)}$ must include a cycle $C^*\in \mathcal{C}$ such that the following two conditions hold: (1) $\mu(C^*)\leq \mu(C_i)$. (2) $\mathcal{B}\setminus\{C_i\}\cup \{C^*\}$ is a homology basis. 

Condition (1) holds because  $\mu(C)\leq \mu(C_i)$ for every cycle $C$ in $B_p^{\mu(C_i)}$. 

For (2), observe that there exists a homology class $[C^*]$ generated by one creator that is independent of homology classes generated by $\mathcal{B}\setminus\{C_i\}$. If no such cycle exists, any homology class generated by one creator of $B_p^{\mu(C_i)}$ can be written as a linear combination of homology classes generated by $\mathcal{B}\setminus\{C_i\}$. The homology classes generated by creators form a homology basis of $B_p^{\mu(C_i)}$ and $C_i\in B_p^{\mu(C_i)}$. It means that $[C_i]$ is not independent of $\mathcal{B}\setminus\{C_i\}$, contradicting the assumption that $\mathcal{B}$ is a homology basis. Therefore, $\mathcal{B}\setminus\{C_i\}\cup \{C^*\}$ is a homology basis.

Combining condition (1) with (2), the homology basis $\mathcal{B}'=\mathcal{B}\setminus\{C_i\}\cup \{C^*\}$ is a minimal homology basis. What is more, if we sorted the cycles in $\mathcal{B}'$ in nondecreasing order of sizes, then the first $i+1$ cycles in $\mathcal{B}'$ are in the candidate set $\mathcal{C}$. This is because the cycle $C^*$ is generated by a creator of $B_p^{\mu(C_i)}$, which means that $C^*\in \mathcal{C}$. Therefore, we find a minimal homology basis all of whose cycles are in the candidate set. \qed
\end{proof}

\subsubsection{Computing a minimal homology basis.}
In this section we discuss an algorithm to find a minimal homology basis from the candidate set. We use annotation, denoted by $a(\cdot)$, to represent and distinguish each cycle. Recall that annotation of a cycle is a $g$-bit vector were $g=rank(\mathsf{H}_d)$, and that two cycles are homologous if and only if their annotations are equal. We first compute the annotations for all $d$-simplices in $\K$ \cite{annotation} and give them a fixed order $\sigma_1, \sigma_2, \cdots, \sigma_{n_d}$ where $n_d$ is the number of $d$-simplices in $\K$. Suppose we sort the cycles in the candidate set in nondecreasing order of their sizes as $C_1, C_2, \cdots, C_{g n_0}$ where $n_0$ is the number of vertices in $\K$. Then, every $d$-cycle $C_i$ in $\K$ can be denoted as $C_i=\sum_{j=1}^{n_d} \gamma_{ij}\sigma_j$ where $\gamma_{ij}\in \{0, 1\}$ and $1\leq i\leq gn_0$. Thus, we have $a(C_i)=\sum_{j=1}^{n_d} \gamma_{ij}a(\sigma_j)$, $1\leq i\leq gn_0$. We compute the annotations $a(C_1), a(C_2), \cdots, a(C_{g n_0})$ for all cycles $C_1, C_2, \cdots, C_{g n_0}$ in the candidate set simultaneously.   

Let 
$
X=
(
a(C_{1})^T ,a(C_{2})^T, \cdots, 
a(C_{g n_0})^T )^T$ and 
$Y=
(a(\sigma_{1})^T, a(\sigma_{2})^T, \cdots, a(\sigma_{n_d})^T
)^T$. 
%\[
%X=
% \left( \begin{array}{c}
%a(C_{1}) \\
%a(C_{2})\\
%\vdots \\
%a(C_{g n_0})
%\end{array} \right);
%Y=
% \left( \begin{array}{c}
%a(\sigma_{1}) \\
%a(\sigma_{2})\\
%\vdots \\
%a(\sigma_{n_d})
%\end{array} \right)\]
The goal is to compute $X$ that satisfies the following equation: $X=\Gamma Y$ where $\Gamma=(\gamma_{ij})_{g n_0\times n_d}$. The computation of the matrix $X$ takes time $O(n^{\omega}g)$ using the fast matrix multiplication algorithm where $\Gamma$ is a $g n_0 \times n_d$ matrix and $Y$ is an $n_d\times g$ matrix. 

Let $X'$ be the transposed matrix of $X$. The problem of computing a minimal homology basis from the candidate set $\mathcal{C}$ is equivalent to computing the earliest basis of the matrix $X'$ \cite{annotation}. %Recall that the matrix $X'$ is a $g \times g n_0$ matrix, hence we can partition $X'$ into submatrices $X'=[X_1|X_2|\cdots]$ where every submatrix $X_i$ has size at most $g\times g$. We first compute the earliest basis of $X_1$ to be $X_J$, then iterate for every submatrix $X_i$. For every $i$, $X_J$ is computed as the early basis of $[X_1|X_2|\cdots|X_{i-1}]$. Then the early basis of $[X_1|X_2|\cdots|X_{i}]$ is the earliest basis of $[X_J|X_i]$. We update $X_J$ to the earliest basis of $[X_J|X_i]$. After iteration, $X_J$ is the earliest basis of $X$. As a result, the corresponding cycles form a minimal homology basis. 
According to Proposition \ref{prop:annotation}, computing the earliest basis of $X'$ costs us $O(ng^{\omega})$ time. Combining the time $O(n^{\omega+1})$ in building the candidate set $\mathcal{C}$ and the time $O(n^{\omega}g)$ in computing $X$, we conclude that the total running time is $O(n^{\omega+1})$.
\begin{theorem}
Given a simplicial complex $\K$ with $n$ simplices, there is an algorithm to compute a minimal homology basis as defined in Definition \ref{def:MCBd} in any dimension in time $O(n^{\omega+1})$.
\end{theorem}

\section{Proof of Claim \ref{claim:SPexists}}
\label{appendix:claim:SPexists}

We prove this claim by induction. 
First, fix any source node $x \in vert(\K)$. We sort all other vertices in $vert(\K)$ in non-decreasing order of $d(x,y)$; that is, 
$d(x, y_1)\le d(x, y_2) \le \ldots, \le d(x, y_{s})$ with $s = |vert(\K)|-1$. 
We carry out an induction based on this order. 
For the base case, any path in (ii) of Definition \ref{def:distance} is necessarily a \myshortestpath{} from $x$ to $y_1$: Indeed, if there is any other vertex $y$ (other than $x$ and $y_1$) in such a path, it is necessary that $d(x,y) = d(x, y_1)$ as $d(x,y_1)$ has the smallest distance to $x$. 

Now suppose there exists a \myshortestpath{} from $x$ to $y_i$ for $1\leq i\leq s$. 
Consider $y_{i+1}$ and assume that there is no \myshortestpath{} from $x$ to $y_{i+1}$. 
By Definition \ref{def:distance}, there exists a path $\Pi=(u_0=x, u_1, \ldots, u_k=y_{i+1})$ such that for every vertex $u_i\in \Pi$, $d(x, u_i)\leq d(x, y_{i+1})$. 
As this path violates the conditions in Definition \ref{def:SPpath}, there must exist a pair of vertices $u_j, u_l\in \Pi$ with $j<l$ such that $d(x, y_{i+1})\geq d(x, u_j)>d(x, u_l)$. Let $l$ be the maximal value with which such a pair $(j, l)$ exists. It follows that we have $d(x, u_l) \le d(x, u_{l+1}) \le \ldots \le d(x, u_k)$. By inductive hypothesis, we know that there is a \myshortestpath{} $\Pi^*$ from $x$ to $u_l$ since $d(x, u_l)<d(x, y_{i+1})$. Hence, the path $\Pi^*$ concatenated with the sub-path of $\Pi$ from $u_l$ to $u_k = y_{i+1}$ gives a \myshortestpath{} from $x$ to $y_{i+1}$. 
The claim thus follows from induction. 

%%%%%%%%%%%%%%%%%%%%%%%%%%%%%%%%%%%%%%%%%%%
\section{Ensuring uniqueness of shortest paths}
\label{appendix:ObtainLex}
\begin{wrapfigure}{r}{5cm}
\centering
\begin{tikzpicture}[scale=0.7]
	\tikzstyle{every node}=[font=\small,scale=0.8]
    \tikzstyle{point}=[{draw, fill=white, circle, inner sep=0.5pt}]
    \node (a1)[point] at (0,0) {};
    \node[left] at (0, 0) {$a_1$};
    
    \node (a4)[point] at (1,-.5) {};
    \node[above] at (1, -.5){$a_4$};
    
    \node (a5)[point] at (1.5, -1) {};
    \node[right] at (1.5, -1) {$a_5$};
    
    \node (a3)[point] at (.5, -1) {};
    \node[left] at (0.5, -1) {$a_3$};
    
    \node (a2)[point] at (2, 0) {};
    \node[right] at (2, 0) {$a_2$};
   
    \draw plot [smooth] coordinates {(a1)(a4)(a5)(a3)(a4)(a2)};
    \end{tikzpicture}
    \caption{Path $\pi_1=a_1a_4a_5a_3a_4a_2$ and $\pi_2=a_1a_4a_3a_5a_2$ are two \myshortestpath{}s from $a_1$ to $a_2$. Consider a new path $\pi=a_1a_4a_2$ which is a \myshortestpath{} from $a_1$ and $a_2$. However $len(\pi)=2<5=len(\pi_1)=len(\pi_2)$.}
    \vspace{-20pt}
  \label{exampleLexPath}
\end{wrapfigure}
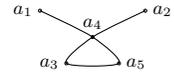
In section \ref{sec:generalization}, we require that the \myshortestpath{} (in this section, we use shortest path for short) in $\mathcal{K}$ between any two vertices is unique. Now we show how to avoid this restriction using an idea from \cite{Hartvigsen1994}. 
\begin{lemma}\label{lex}
 Let $\mathcal{K}$ be a simplicial complex with a path-dominated distance $d(\cdot, \cdot)$. For every pair of nodes, there exists a unique shortest path $\pi$ from $u$ to $v$ that satisfies exactly one of the following two conditions w.r.t. any other path $\pi'$ from $u$ to $v$:\\
(1) $len(\pi)<len (\pi')$\\
(2) $len(\pi)=len(\pi'), min(vert(\pi)\setminus vert(\pi')) < min(vert(\pi')\setminus vert(\pi))$\\
Here $len(\pi)$ denotes the number of edges in a path $\pi$ and $min(U)$ denotes the minimum index of the vertices in a subset $U$ of $vert(\K)$. We say $\pi<\pi'$ if the above two conditions hold.
\end{lemma}

The proof follows from \cite[Proposition 4.1]{Hartvigsen1994}.
%\begin{proof}
%See Figure \ref{exampleLexPath}. Let \textbf{P} be the set of all shortest paths from $u$ to $v$ such that $len(\pi^*)$, $\pi^*\in \textbf{P}$, is minimum. It suffices to show that for every $\pi_1, \pi_2\in \textbf{P}$, $vert(\pi_1)\neq vert(\pi_2)$, i.e. $vert(\pi_1) \setminus vert(\pi_2)\neq\varnothing$ and $vert(\pi_2) \setminus vert(\pi_1)\neq\varnothing$, since $\pi_1$ and $\pi_2$ have the same number of edges and vertices. Assume that $vert(\pi_1)=vert(\pi_2)$. Let us rename the nodes so that $\pi_1$ contains the nodes $u_1=u, u_2, \cdots, u_k=v$ in this order and $\pi_2$ contains the nodes $v_1=u, v_2, \cdots, v_k=v$ in this order. Since $\pi_1\neq \pi_2$, the sequences $u_1,\cdots, u_k$ and $v_1,\cdots, v_k$ are not equal. Let $j$ be the smallest index such that $u_j\neq v_j$. Hence in $\pi_2$, $u_j$ must appear after $v_j$ in the sequence $v_1,\cdots, v_k$, which means that there is a $l>j$ with $v_l=u_j$. By definition of shortest path, $d(u, v_j)\leq d(u, u_j)$.
%Consider the path $\pi'$ that consists of a subpath $\pi'_1$ of $\pi_1$ from $u_1=u$ to $u_j=v_l$ and a subpath $\pi'_2$ of $\pi_2$ from $v_l=u_j$ to $v_k=v$. It is easy to see that $\pi'$ is also a shortest path from $u$ to $v$ by definition. The way we choose the vertex $u_j$ suggests that $len(\pi'_2)$ is smaller than the number of vertices in $\pi_1$ from $u_j$ to $u_k$, which means that $len(\pi')<len(\pi_1)=len(\pi_2)$ reaching a contradiction.\qed
%\end{proof}

We now describe the algorithm to compute a shortest path tree $T_p$ w.r.t. a path-dominated distance $d(\cdot, \cdot)$ rooted at $p$ under the uniqueness condition. Let $\pi_p(q)$ be the tree path from $p$ to $q$ in the current partial tree. Initially we set a priority queue $Q$ the vertex set $vert(\K)$. Every time we delete a vertex $q$ in the queue $Q$ with the least distance $d(p, q)$, least value $len_p(q)$ and least index. We iterate for all neighbors $w$ of $q$: If $\pi_p(q)\circ e$, $e=(q,w)$, is a shortest path from $p$ to $w$, and is smaller than $\pi_p(w)$ as in Lemma \ref{lex}
%we consider the two conditions in Lemma \ref{lex} to find the smaller one of $\pi_p(w)$ and $\pi_p(q)\circ e$, 
we will update the tree path $\pi_p(w)$ in $T_p$ as $\pi_p(q)\circ e$. Note that those vertices not in $Q$ will not be updated. Hence there are $O(n)$ iterations. 
%Let $len_p(u)$ be the number of edges in the tree path from $p$ to $u$ and $parent_p(u)$ be the parent of $u$ in tree $T_p$. Let $\pi_p(u, v)$ denote the unique path in tree $T_p$ from $u$ to $v$. We check for all neighbors $w$ of $q$ for the two conditions in Lemma \ref{lex}. Note that those vertices not in $Q$ will not be updated. Hence there are $O(n)$ iterations. 
%\begin{algorithm}
%\caption{ShortestPathTree}\label{UniqueShortestPathTree}
% \begin{algorithmic}[1] 
%	\State Priority queue $Q\gets vert(\K)$; $len_p(p)\gets 0$, $len_p(u)\gets \infty$, $u\neq p$; $parent_p(v)\gets 0$, $\forall v\in vert(\mathcal{K})$
%	\While{$Q\neq \varnothing$}
%		\State $q=Q.delMin()$
%		\For{every neighbor $w$ of $q$}
%			\If{$d(p, q)\leq d(p, w)$}
%				\State $r=parent_p(w)$
%				\If{$len_p(q)+1<len_p(w)$}
%					\State $parent_p(w)=q$
%					\State $len_p(w)=len_p(q)+1$
%				\Else
%					\If{$len_p(q)+1=len_p(w)$ and $min(vert(\pi_p(p, q)\circ (q,w))\setminus vert(\pi_p(p, r)\circ (r,w)))<min(vert(\pi_p(p, r)\circ (r,w))\setminus vert(\pi_p(p, q)\circ (q,w)))$}
%						\State $parent_p(w)=q$
%						\State $len_p(w)=len_p(q)+1$
%						\EndIf
%				\EndIf
%	\EndIf
%	\EndFor
%	\EndWhile\\
%	\Return $parent_u$
%	\end{algorithmic}
%\end{algorithm}
What remains is to compute the minimum index of $vert(\pi)\setminus vert(\pi')$ given two tree paths $\pi$ and $\pi'$ from $p$ to any vertex $v$. This can be achieved in time $O(\log n)$ adapting the algorithm from~\cite{wulff2009minimum} for \myshortestpath{}. 

Thus we conclude the above analysis with the following theorem.
\begin{theorem}
The shortest path tree in a simplicial complex $\mathcal{K}$ can be computed in $O(n\log n)$ time.
\end{theorem}
%%%%%%%%%%%%%%%%%COUNTER-EXAMPLE%%%%%%%%%%%%%%%%
%\section{A counterexample}\label{appendix:counterexample}

\end{document}